\documentclass{article}

\pdfoutput=1        

\usepackage{amsmath}
\usepackage{amsthm}
\usepackage{amsfonts}
\usepackage{float}
\usepackage[dvips,final]{graphicx}
\usepackage{dcolumn}
\usepackage{comment} 
\usepackage{xspace}
\usepackage{url}
\usepackage{upgreek}
\usepackage{multicol}
\usepackage{pict2e,multirow}
\usepackage{verbatim}
\usepackage{wrapfig}
\usepackage{url}
\usepackage{tikz}
\usepackage[margin=0.35in]{caption}
\usepackage[colorlinks=true,citecolor=blue,urlcolor=blue]{hyperref}
\usepackage{subfigure}
\usepackage[T1]{fontenc}
\usepackage{titling}
\usepackage[hmargin=.720in]{geometry}

\newcommand{\sspcoef}{{\cal{C}}}
\newcommand{\dx}{\Delta x}
\newcommand{\dt}{\Delta t}
\newcommand{\Dt}{\Delta t}
\newcommand{\m}[1]{\mathbf{#1}}
\newcommand{\mD}{\m{D}}
\newcommand{\DtFE}{\dt_{\textup{FE}}}

\newcommand\vy{{\bf y}}
\newcommand\ve{{\bf e}}

\newcommand{\aij}{\alpha_{i,j}}
\newcommand{\bij}{\beta_{i,j}}

\newtheorem{rmk}{Remark}
\newtheorem{thm}{Theorem}

\title{Explicit Strong Stability Preserving  Multistage Two-Derivative Time-Stepping Schemes}
\author{Andrew J. Christlieb$^{1}$, 
Sigal Gottlieb$^2$,
Zachary Grant$^2$\thanks{Corresponding author: zgrant@umassd.edu}, David C. Seal$^3$\\
{\small $^1$Department of Computational Mathematics Science and Engineering, 
Department of Electrical Engineering,} \\
{\small and Department of Mathematics,
Michigan State University} \\
{\small $^2$Department of Mathematics, University of Massachusetts, Dartmouth} \\
{\small $^3$Department of Mathematics, U.S. Naval Academy.}
}
\date{}
\begin{document}

\maketitle
\bibliographystyle{siam}

\vspace{-0.5in}
\abstract{High order strong stability preserving (SSP)  time discretizations are advantageous 
for use with  spatial discretizations with nonlinear stability properties for the
solution of hyperbolic PDEs. The search for high order strong stability time-stepping methods 
with large allowable strong stability time-step has been an active area of
research over the last two decades. Recently, multiderivative time-stepping 
methods have been implemented with hyperbolic PDEs. In this work we describe sufficient 
conditions for a two-derivative multistage method to be SSP, and find
some optimal SSP multistage two-derivative methods. While 
explicit SSP Runge--Kutta methods exist only up to fourth order, we show that this order barrier is broken
for explicit multi-stage two-derivative methods by designing a three stage fifth order SSP method. These methods
are tested on simple scalar PDEs to verify the order of convergence, and demonstrate the need for the SSP condition and the
sharpness of the SSP time-step in many cases.}


\vspace{-0.1in}
\section{Introduction}
\vspace{-0.05in}
\subsection{SSP methods}
When numerically approximating the solution to a hyperbolic conservation law
of the form 
\begin{eqnarray}
\label{pde}
	U_t + f(U)_x = 0,
\end{eqnarray}
difficulties arise when 
the exact solution develops sharp gradients or discontinuities. 
Significant effort has been expended on developing spatial discretizations
that can handle discontinuities \cite{SSPbook2011}, especially for high-order methods. These discretizations
have special nonlinear non-inner-product stability properties, such as total variation stability or positivity, 
which ensure that when the semi-discretized equation
\begin{eqnarray}
\label{ode}
u_t = F(u),
\end{eqnarray}
(where $u$ is a vector of approximations to $U$) 
is evolved using a forward Euler method, the numerical
solution satisfies  the desired strong stability property,
\begin{equation} \label{eqn:FEcond}
	\| u^n + \dt F(u^{n})  \| \leq \| u^n \|, \quad  0  \leq \dt \leq \Delta t_{FE},
\end{equation}
where $\| \cdot \|$ is any desired norm, semi-norm, or convex functional.

In place of the first order time discretization \eqref{eqn:FEcond}, we typically require a 
higher-order time integrator, 
but we still wish to ensure that the  strong stability  property 
$	\| u^{n+1} \| \le \|u^n\| $
is satisfied, perhaps under a modified time-step restriction, where $u^n$ is a discrete approximation
to $U$ at time $t^n$. 
In \cite{shu1988b}  it was  observed that some Runge--Kutta methods can be decomposed into convex combinations
of forward Euler steps, so that any  convex functional  property satisfied by \eqref{eqn:FEcond} will be {\em preserved}
by these higher-order time discretizations. 
For example, the $s$-stage explicit Runge--Kutta method   \cite{shu1988},
\begin{eqnarray}
\label{rkSO}
y^{(0)} & =  & u^n, \nonumber \\
y^{(i)} & = & \sum_{j=0}^{i-1} \left( \aij y^{(j)} +
\dt \bij F(y^{(j)}) \right), \; \; \; \; i=1, . . ., s\\
 u^{n+1} & = & y^{(s)}  \nonumber
\end{eqnarray}
can be rewritten as  convex combination of forward Euler steps of the form \eqref{eqn:FEcond}.
If all the coefficients $\aij$ and $\bij$ are non-negative, and provided $\aij$ is zero only if 
its corresponding $\bij$ is zero, then each stage is bounded by
\begin{eqnarray*}
\| y^{(i)}\| & =  & 
\left\| \sum_{j=0}^{i-1} \left( \aij y^{(j)} + \dt \bij F(y^{(j)}) \right) \right\|   
 \leq   \sum_{j=0}^{i-1} \aij  \, \left\| y^{(j)} + \dt \frac{\bij}{\aij} F(y^{(j}) \right\|  .
\end{eqnarray*}
Noting that each $\| y^{(j)} + \dt \frac{\bij}{\aij} F(y^{(j)}) \| \leq \| y^{(j)} \|$ for
$ \frac{\bij}{\aij}  \dt \leq \DtFE$, and by consistency $\sum_{j=0}^{i-1} \aij =1$,
we have $ \| u^{n+1}\| \leq \| u^{n}\| $ as long as  
\begin{eqnarray}
	\label{rkSSP}
	\dt \leq \sspcoef  \DtFE \; \; \; \; \forall i, j,
\end{eqnarray}
where $\sspcoef = \min  \frac{\aij}{\bij}$.
(We employ the convention that if any of the $\beta$'s are equal to zero, the corresponding 
ratios are considered infinite.)  The resulting time-step restriction is a combination of two
distinct factors: (1) the term $\DtFE$ that depends on the spatial discretization, and 
(2) the SSP coefficient $\sspcoef$ that depends only on the time-discretization. Any method
that admits such a decomposition with $\sspcoef>0$ is called a {\em strong stability preserving (SSP)}
method.

This convex combination  decomposition was used in the development of 
 second and third order explicit Runge--Kutta methods \cite{shu1988} 
 and later of fourth order methods \cite{SpiteriRuuth2002, ketcheson2008}
that guarantee the  strong stability properties of any  spatial discretization, 
provided only that these properties are satisfied when using the forward Euler (first derivative) condition
in \eqref{eqn:FEcond}.
Additionally, the convex combination approach  also guarantees that 
the intermediate stages in a Runge--Kutta method  satisfy the strong stability property as well.

The convex combination approach clearly provides a sufficient condition for preservation of strong stability.
Moreover, it has also be shown that this condition is necessary
 \cite{ferracina2004, ferracina2005,higueras2004a, higueras2005a}.
Much research on SSP methods focuses on finding high-order time discretizations
with the largest allowable time-step  $\Dt \le \sspcoef \DtFE$ by maximizing
the   {\em SSP coefficient} $\sspcoef$ of the method. 
It has been shown that explicit  Runge--Kutta methods with positive SSP coefficient cannot be more than
fourth-order accurate \cite{kraaijevanger1991,ruuth2001}; this led to the 
study of other classes of explicit SSP methods, such as methods with multiple steps.
Explicit multistep SSP methods of order  $p>4$ do exist, but have 
severely restricted time-step requirements \cite{SSPbook2011}. Explicit multistep
multistage methods that are SSP and have order $p>4$ have been developed as well \cite{tsrk,msrk}.

Recently, multi-stage multiderivative methods have been proposed for use with hyperbolic PDEs
\cite{sealMSMD2014, tsai2014}.
The question then arises as to whether these methods can be strong stability preserving as well.
 Nguyen-Ba and colleagues studied the SSP properties of the Hermite-Birkoff-Taylor methods 
with a set of simplified base conditions in  \cite{Nguyen-Ba2010}.
In this work we consider multistage two-derivative methods and develop sufficient conditions for strong stability preservation 
for these methods, and we show that explicit SSP methods within this class can break this well-known order barrier for 
explicit Runge--Kutta methods. Numerical results demonstrate that the SSP condition is useful in preserving the 
nonlinear stability properties of the underlying spatial discretization and that the allowable time-step predicted by the SSP
theory we developed is sharp in  many cases.

\vspace{-.125in}
\subsection{Multistage multiderivative methods} \label{MSMDoc}
To increase the possible order of any method, we can use more steps (e.g. linear multistep methods), 
more stages (e.g. Runge--Kutta methods), or more derivatives (Taylor series methods).
It is also possible to combine these approaches to obtain methods with multiple steps, 
stages, and derivatives. Multistage multiderivative integration methods were first considered in 
\cite{obreschkoff1940,Tu50,StSt63}, and multiderivative time integrators for ordinary differential 
equations have been developed  in \cite{shintani1971,shintani1972,KaWa72,KaWa72-RK,
mitsui1982,ono2004, tsai2010},  but only recently have these methods been 
explored for use with  partial differential equations (PDEs) \cite{sealMSMD2014,tsai2014}.
 In this work, we consider 
explicit multistage two-derivative time integrators as applied to the numerical solution of 
hyperbolic conservation laws.

We consider the system of ODEs \eqref{ode} resulting from the spatial discretization of a 
hyperbolic PDE of the form \eqref{pde}. 
We define the one-stage, two-derivative building block method
$ u^{n+1} = u^n + \alpha \Delta t F(u^n) + \beta \Delta t^2 \dot{F}(u^n) $
where $\alpha \geq 0$ and $\beta \geq 0$ are coefficients chosen to ensure the desired order. This method
can be at most second order, with coefficients $\alpha=1$ and $\beta = \frac{1}{2}$. This is 
the second-order Taylor series method. 
To obtain higher order explicit methods, we can add more stages:
{\small
\begin{eqnarray}
\label{MSMD}
y^{(i)} & = &  u^n +  \dt \sum_{j=1}^{i-1} \left( a_{ij} F(y^{(j)}) +
\dt \hat{a}_{ij} \dot{F}(y^{(j)}) \right), \; \; \; \; i=1, . . ., s \\
 u^{n+1} & = &  u^n +  \dt \sum_{j=1}^{s} \left( b_{j} F(y^{(j)}) +
\dt \hat{b}_{j} \dot{F}(y^{(j)}) \right) . \nonumber
\end{eqnarray}
We can write the coefficients in matrix vector form, where
\[ A = \left( \begin{array}{llll}
 0 & 0 &  \vdots & 0 \\
 a_{21} & 0 & \vdots & 0  \\
 \vdots & \vdots & \vdots & \vdots \\
 a_{s1} & a_{s2} & \vdots & 0\\
 \end{array} \right) , 
 \; \; 
 \hat{A} = \left( \begin{array}{llll}
 0 & 0 &  \vdots & 0 \\
 \hat{a}_{21} & 0 & \vdots & 0  \\
 \vdots & \vdots & \vdots & \vdots \\
 \hat{a}_{s1} & \hat{a}_{s2} & \vdots & 0 \\
 \end{array} \right) , \; \; 
b = \left(  \begin{array}{llll}  b_1 \vspace{0.1in} \\ b_2  \vspace{0.1in}
 \\ \vdots  \\ b_s \vspace{0.1in} \\ \end{array} \right) , \; \;
\hat{b} =\left(\begin{array}{llll}  \hat{b}_1 \vspace{0.1in} \\  \hat{b}_2  \vspace{0.1in}\\ \vdots \\  \hat{b}_s \vspace{0.1in} \\ \end{array} \right) .
  \] 
  We let $c = A \ve$ and $\hat{c} = \hat{A} \ve$, where $\ve$ is a vector of ones.
These coefficients are then selected to attain the desired order, based on the order conditions 
written in Table \ref{MSMD_OC} as described in \cite{tsai2010,GeWi86}.
\begin{table}[!b] 	\hspace{0.7in} \def\arraystretch{1.4} {\small
\noindent \begin{tabular}{|l|l|} \hline
	 $p = 1$  & $b^T e =1$  \\ \hline
	 $p = 2$  &  $b^T c+\hat{b}^Te =\frac{1}{2}$ \\ \hline
          $p=  3$ & $b^T c^2 + 2\hat{b}^T c=\frac{1}{3}$ \\ 
	    & $b^TAc+b^T\hat{c}+\hat{b}^Tc=\frac{1}{6}$ \\ \hline
	  $p=4$   & $b^Tc^3+3\hat{b}^Tc^2=\frac{1}{4}$ \\
	 & $b^TcAc+b^Tc\hat{c}+\hat{b}^Tc^2+\hat{b}^TAc+\hat{b}^T\hat{c}=\frac{1}{8}$ \\
	 & $b^TAc^2+2b^T\hat{A}c+\hat{b}^Tc^2=\frac{1}{12}$ \\
	 &  $b^TA^2c+ b^TA\hat{c}+ b^T\hat{A}c+\hat{b}^TAc+\hat{b}^T\hat{c}=\frac{1}{24}$ \\ \hline
	$p = 5$   &      $ b^Tc^4 + 4\hat{b}^Tc^3 =\frac{1}{5}$ \\
	           &  $b^Tc^2Ac + b^Tc^2\hat{c}+\hat{b}^Tc^3+2\hat{b}^TcAc+2\hat{b}^Tc\hat{c}=\frac{1}{10}$ \\
	           &  $b^TcAc^2+2b^Tc\hat{A}c+\hat{b}^Tc^3+\hat{b}^T Ac^2+2\hat{b}^T\hat{A}c=\frac{1}{15}$ \\
	           &  $b^TcA^2c+b^TcA\hat{c}+b^Tc\hat{A}c+\hat{b}^TcAc+\hat{b}^Tc\hat{c} 
	         + \hat{b}^TA^2c+\hat{b}^TA\hat{c}+\hat{b}^T\hat{A}c=\frac{1}{30}$ \\
	           &  $b^T(Ac)(Ac)+2b^T\hat{c}Ac+b^T\hat{c}^2+ 2\hat{b}^TcAc+2\hat{b}^Tc\hat{c}=\frac{1}{20}$ \\
	          & $ b^TAc^3+3b^T\hat{A}c^2+\hat{b}^Tc^3=\frac{1}{20}$\\
	           &  $b^TA(cAc)+b^TA(c\hat{c})+b^T\hat{A}c^2+b^T\hat{A}Ac+b^T\hat{A}\hat{c}  
	     + \hat{b}^TcAc+\hat{b}^Tc\hat{c}=\frac{1}{40}$ \\
	 &  $b^TA^2c^2+2b^TA\hat{A}c+b^T\hat{A}c^2+\hat{b}^T Ac^2+2\hat{b}^T\hat{A}c=\frac{1}{60}$ \\
	 &  $b^TA^3 c+b^TA^2 \hat{c}+b^T A \hat{A}c+b^T\hat{A}Ac+b^T\hat{A}\hat{c}
	+\hat{b}^TA^2c +\hat{b}^TA\hat{c}+\hat{b}^T\hat{A}c=\frac{1}{120} $ \\ \hline
	\end{tabular}}

	\caption{Order conditions for multi-stage multiderivative methods of the form \eqref{MSMD} as in \cite{tsai2010}.}
	\label{MSMD_OC} 
\end{table}

\begin{rmk}
In this work, we focus on multistage two-derivative methods as time integrators for use with hyperbolic PDEs.
In this setting, the operator  $F$  is obtained by a spatial discretization of the term $U_t= -f(U)_x$ to obtain the 
system $u_t = F(u)$. This is the typical method-of-lines approach, and SSP methods were introduced in the context of this
approach. The computation of the second derivative term $\dot{F}$ should follow directly from the definition of $F$,
where we compute $\dot{F} = F(u)_t = F_u u_t = F_u F$. In practice, the calculation of $F_u$ may be computationally
prohibitive, as for example in the popular WENO method where $F$ has a highly nonlinear dependence on $u$.

Instead, we adopt a Lax-Wendroff type approach, where we use the fact that the system of ODEs arises from the PDE
\eqref{pde} to replace the time derivatives by the spatial derivatives, and discretize these in space. This approach 
begins with the observation that $F(u) = u_t = U_t + O(\Delta x^m)$ (for some integer $m$). The term $F(u)$ is 
typically computed  using a conservative spatial 
discretization $D_x$ applied to the flux: 
\[F(u) = D_x(-f(u)).\]
Next we approximate 
\[F(u)_t = u_{tt} \approx U_{tt} = -f(U)_{xt} = \left(-f(U)_t \right)_x =   
\left(-f'(U) U_t \right)_x \approx 
\tilde{D}_x \left( -f'(u) u_t \right),
\]
where a (potentially different) spatial approximation $\tilde{D}_x$ is used.
This means that
\[F(u)_t = u_{tt} =  U_{tt} + O(\Delta x^n) = \dot{F} + O(\Delta x^r) \] (for some integers $n$ and $r$). 

Since $F_t  = \dot{F} + O(\Delta x^r)$, we will not necessarily obtain the time order $\Delta t^p$ when we satisfy the 
order conditions in Table \ref{MSMD_OC}, as our temporal order is polluted by spatial errors as well. 
The concern about the Lax-Wendroff approach is that the order conditions in Table  \ref{MSMD_OC} are based on the assumption that
 $\dot{F} = F_t $, which is not exactly correct in our case, thus introducing additional errors. However, these errors 
are of  order $r$ in space, so that in practice, as long as the 
 spatial errors are smaller than the temporal errors, we expect to see the correct order of accuracy in time. 

To verify this, in Section \ref{convergence} 
we perform numerical convergence studies of these temporal methods where $\dot{F}$ is approximated by high order
spatial schemes and compare the errors from these methods to those from  
well-known SSP Runge--Kutta methods. We observe that if the spatial and temporal grids are refined together, the expected order of 
accuracy is demonstrated. Furthermore, if the spatial grid is held fixed but the spatial discretization is highly accurate, the correct
time-order is observed until the time error falls below the spatial error.
We conclude that in practice, it is not necessary to compute $F_t$ exactly, and that the use of a Lax-Wendroff type procedure
that replaces the temporal derivatives by spatial derivatives and discretizes each of these independently, does not destroy the 
temporal accuracy.

This observation is not new: many other methods have adopted this type of approach and obtained genuine high-order accuracy.
Such methods include the original Lax-Wendroff method \cite{LW1960},  ENO methods  \cite{harten1987b},
finite volume ADER methods \cite{ADER}, the finite difference WENO Schemes in \cite{ShuQiu}, and the Lax-Wendroff discontinuous 
Galerkin schemes \cite{QiuDumbserShu2005}.
\end{rmk}
In the next section we will discuss how to ensure that a multistage two-derivative method will preserve these strong stability
properties. 
	
\section{The SSP condition for multiderivative methods}
\subsection{Motivating Examples}
To understand the strong stability condition for multiderivative methods, we consider the strong stability properties of a 
multiderivative building block of the form
\[ u^{n+1} = u^n + \alpha \Delta t F(u^n) + \beta \Delta t^2 \dot{F}(u^n) ,\]
and begin with the simple linear one-way wave equation $U_t = U_x$.
This equation has the property that its second derivative in time
is, with the assumption of sufficient smoothness, also the second derivative in space:
\[
	U_{tt} = (U_x)_t = ( U_{t} )_x = U_{xx}.
\]
We will use this convenient fact in a Lax-Wendroff type approach to define $\dot{F}(u^n)$ by a spatial discretization of $U_{xx}$.

For this problem, we define $F$ by the original first-order upwind method 
\begin{subequations}
\label{eqn:low-order}
\begin{equation}
\label{eqn:low-order-a}
	F(u^n)_j := \frac{1}{\dx} \left(u^n_{j+1} - u^n_j \right) \approx U_x( x_j ),
\end{equation}
and $\dot{F}$ by the second order centered discretization to $U_{xx}$:
\begin{equation}
\label{eqn:low-order-b}
	\dot{F}(u^n)_j := \frac{1}{\dx^2} \left( u^n_{j+1}- 2 u^n_j + u^n_{j-1} \right) \approx  U_{xx}( x_j ).
\end{equation}
\end{subequations}
These spatial discretizations are total variation diminishing (TVD) in the following sense:
\begin{subequations}
\label{eqn:tvd-properties}
\begin{eqnarray}
\label{FEex}
u^{n+1} &=& u^n + \Delta t F(u^n)   \; \;  \; \;  \mbox{is TVD for} \; \; \; \; \Delta t \leq \Delta x,  \\
\label{VVex}
u^{n+1} &=& u^n + \Delta t^2 \dot{F}(u^n)  \; \; \; \;  \mbox{is TVD for} \; \; \; \;   \Delta t \leq  \frac{\sqrt{2}}{2} \Delta x.
 \end{eqnarray}
\end{subequations}

\begin{rmk} Note that we chose a second derivative $\dot{F}$ in space that is not an \emph{exact} derivative $F_t(u^n)$ of 
$F(u^n)$ in the method-of-lines formulation. However, as noted in Remark 1 and will be shown in the convergence studies, 
if the spatial and temporal grids are co-refined, 
this provides a sufficiently accurate approximation to $\dot{F}(u^n)$.
The exact derivative can be obtained by applying the upwind differentiation operator to the
solution twice, which produces 
\[ F_t  := \frac{1}{\dx^2} \left( u^n_{j+2}- 2 u^n_{j+1} + u^n_{j} \right) .\] However, computing $\dot{F}=F_t$ using this
formulation  does not satisfy the condition
\eqref{VVex} for any value of $\Delta t$.
\end{rmk}

To establish the TVD properties of the multiderivative building block we decompose it:
\begin{eqnarray*}
 u^{n+1} &=& u^n + \alpha \Delta t F(u^n) + \beta \Delta t^2 \dot{F}(u^n) 
= a u^n + \alpha \Delta t F(u^n) + (1-a) u^n + \beta \Delta t^2 \dot{F}(u^n) \\
 & = & a \left( u^n + \frac{\alpha}{a} \Delta t F(u^n) \right) +
 (1-a) \left( u^n + \frac{\beta}{1-a} \Delta t^2 \dot{F}(u^n) \right).
\end{eqnarray*}
It follows that for any $0 \leq a \leq 1$ this is a convex combination of terms of the form \eqref{FEex} and \eqref{VVex}, and so
\begin{eqnarray*}
 \|u^{n+1} \|_{TV} & \leq &  a \left\| \left( u^n + \frac{\alpha}{a} \Delta t F(u^n) \right) \right\|_{TV} +
(1-a)  \left\|  \left( u^n + \frac{\beta}{1-a} \Delta t^2 \dot{F}(u^n) \right) \right\|_{TV} \\
& \leq & a \left\| u^n \right\|_{TV} + (1-a) \left\| u^n \right\|_{TV}
 \leq  \|u^{n} \|_{TV}
 \end{eqnarray*}
 for time-steps satisfying 
$\Delta t \leq \frac{a}{\alpha} \Delta x $ and
$\Delta t^2 \leq \frac{1-a}{2 \beta}  \Delta x^2. $ 
The first restriction relaxes as  $a$ increases while the second 
becomes tighter as $a$ increases, so that the value of $a$ that maximizes these conditions 
occurs when these are equal.  This is given by
\[
	a^2 + \frac{\alpha^2}{2 \beta} a - \frac{\alpha^2}{2 \beta} =0 ,\; \; \;  \implies
	\; \; \; a =  \frac{ \alpha \sqrt{\alpha^2 + 8 \beta} - \alpha^2}{4 \beta}.\]
Using this SSP analysis, we conclude that 
\begin{eqnarray} \label{BBex}
\left\| u^n + \alpha \Delta t F(u^n)+ \beta  \Delta t^2 \dot{F}(u^n) \right\|_{TV} \leq \left\| u^n \right\|_{TV} 
\; \; \;  \mbox{for} \; \; \; 
 \Delta t \leq \frac{ \sqrt{\alpha^2 + 8 \beta} - \alpha}{4 \beta} \Delta x.
 \end{eqnarray}

Of course, for this simple example we can directly compute the value of $\Delta t$ for which the
multiderivative building block is TVD.  That is, with $\lambda := \frac{\dt}{\Delta x} \geq 0$, we observe that
\begin{equation*}
\|u^{n+1} \|_{TV} = \left\|   \left(  (1 - \alpha \lambda - 2 \beta \lambda^2) u^n_j + (\alpha \lambda +
    \beta \lambda^2 ) u^n_{j+1} + \beta \lambda^2 u^n_{j-1}  \right)    \right\|_{TV}
    \leq \|u^{n} \|_{TV},
\end{equation*}
provided that
\[ 
	1 - \alpha \lambda - 2 \beta \lambda^2  \geq 0  
	\iff \lambda \leq \frac{ \sqrt{\alpha^2 + 8 \beta} - \alpha}{4 \beta}.
\] 
We see that for this case, the SSP bound is sharp:  the convex combination approach 
provides us exactly the same bound as directly computing the requirements for total variation.  

We wish to generalize this for cases in which the second derivative condition \eqref{VVex} holds for 
$ \Delta t \leq  K \DtFE$ where $K$ can take on any positive value, not just $\frac{\sqrt{2}}{2} $.
For the two-derivative  building block method this can be done quite easily:
\begin{thm}
Given $F$ and $\dot{F}$ such that
\[ \| u^n + \Delta t F(u^n) \| \leq \|u^n\|  \; \; \; \;  \mbox{for} \; \; \; \; \Delta t \leq \DtFE,  \]
and
\[ \| u^n + \Delta t^2 \dot{F}(u^n)  \| \leq \|u^n\|
  \; \; \; \;  \mbox{for} \; \; \; \;   \Delta t \leq  K \DtFE ,\]
 the two-derivative building block 
\[u^{n+1} = u^n + \alpha \Delta t F(u^n) + \beta \Delta t^2 \dot{F}(u^n) \] satisfies the monotonicity condition
 $ \|u^{n+1} \| \leq \|u^n \| $ under the time-step restriction
 \[\Delta t \leq  \frac{K}{2 \beta} \left( 
 \sqrt{ \alpha^2 K^2 + 4 \beta} - \alpha K 
 \right) \DtFE. \]
\end{thm}
\begin{proof}
As above, we rewrite 
\begin{eqnarray*}
u^{n+1} &= & a \left( u^n + \frac{\alpha}{a} \Delta t F(u^n) \right) +
 (1-a) \left( u^n + \frac{\beta}{1-a} \Delta t^2 \dot{F}(u^n) \right),
\end{eqnarray*}	
which is a convex combination provided that $0 \leq a \leq 1$. The time-step restriction that follows from this 
convex combination must satisfy
 $ \frac{\alpha}{a} \Delta t  \leq \DtFE$ and $\frac{\beta}{1-a} \Delta t^2 \leq K^2 \DtFE^2$.
The first condition becomes $ \Delta t  \leq \frac{a}{\alpha} \DtFE$ while the second is 
$ \Delta t  \leq  \sqrt{\frac{1-a}{\beta} } K \DtFE$. 
We observe that on $0 \leq a \leq 1$
the first term encourages a larger $a$ while the second term 
is less restrictive with a smaller $a$. 
The two conditions are balanced, and thus the allowable time-step is maximized, when we equate the right hand sides:
\[\frac{a}{\alpha} =  \sqrt{\frac{1-a}{\beta} } K \; \; \; \; \rightarrow  \; \; \; \; 
 a = \frac{\alpha K}{2 \beta} \left( 
 \sqrt{ \alpha^2 K^2 + 4 \beta} - \alpha K 
 \right). \]
 Now using the first condition, $ \Delta t  \leq \frac{a}{\alpha} \DtFE$ we obtain our result.
\end{proof}

A more realistic motivating example is the unique two-stage fourth order  method
\begin{eqnarray}\label{2s4p}
    u^*     &=& u^n + \frac{\Delta t}{2} F(u^n) + \frac{\Delta t^2}{8} \dot{F}(u^n),  \nonumber \\
    u^{n+1} &=& u^n + \Delta t           F(u^n) + \frac{\Delta t^2}{6}( \dot{F}(u^n)+2\dot{F}(u^*)).
\end{eqnarray}
The first stage of the method is a Taylor series method with $\frac{\Delta t}{2}$, while the second stage can be written
\begin{eqnarray*}
u^{n+1} &= & u^n + \Delta t           F(u^n) + \frac{\Delta t^2}{6}( \dot{F}(u^n)+2\dot{F}(u^*))\\
& = & a \left( u^* - \frac{\Delta t}{2} F(u^n) -  \frac{\Delta t^2}{8} \dot{F}(u^n) \right) + (1-a) u^n 
+ \Delta t           F(u^n) + \frac{\Delta t^2}{6}( \dot{F}(u^n)+2\dot{F}(u^*))\\
& = &  (1-a) \left( u^n  
	+ \frac{1 - \frac{a}{2}}{1-a}  \Delta t   F(u^n) 
	+   \frac{\frac{1}{6} -  \frac{a}{8}}{1-a}  \Delta t^2 \dot{F}(u^n) \right)
	+ a \left( u^* + \frac{\Delta t^2}{3a} \dot{F}(u^*)  \right) .
\end{eqnarray*}
 For $ 0 \leq a \leq 1$ this is a convex combination of two terms.
The first term  is of the form \eqref{BBex}, which gives the time-step restriction
\[ \dt \leq \frac{6}{4-3a} \left( \sqrt{ \left( \frac{5}{4} a^2 - \frac{10}{3} a +
 \frac{7}{3} \right) } - 1 + \frac{a}{2}  \right) \DtFE. \]
The second is of the form \eqref{VVex}, so we have
$ \dt \leq  \sqrt{\frac{3 a}{2} } \DtFE.$
We plot these two in Figure \ref{Fig2s4pSSP}, and we observe
that the first term is decreasing in $a$ (blue line) while the second term is increasing in $a$ (red line).
As a result, we obtain the optimal allowable time-step by setting these two equal, which yields
$a \approx 0.3072182638002141$ and the corresponding SSP coefficient,
$\sspcoef \approx 0.6788426884782078$. A direct computation of the TVD time-step for this case,
which takes advantage of the linearity of the problem and the spatial discretization, gives the bound
$\dt \leq (\sqrt{3}-1) >  0.6788.$ This shows that, as we expect, the SSP condition is not always sharp.

The convex combination approach becomes more complicated when dealing with multi-stage methods. It is the 
most appropriate approach  for developing an understanding of the strong stability property
of a given method. However, it is not computationally efficient for finding optimal SSP methods.
In the following section, we show how to generalize the convex combination decomposition approach 
and we use this generalization to formulate an SSP optimization problem along the lines of \cite{ketcheson2008,ketchcodes}.

\subsection{Formulating the SSP optimization problem}

As above, we begin with the hyperbolic conservation law \eqref{pde}
and adopt a spatial discretization so that we have the system of ODEs \eqref{ode}. 
The spatial discretization $F$ is specially designed so that it satisfies 
the {\em forward Euler} (first derivative) condition
\begin{eqnarray} \label{FE}
	\mbox{\bf Forward Euler condition} \; \; \; \; \; \; \; \;  \|u^n +\Delta t F(u^n) \| \leq \| u^n  \| 
	\; \; \; \; \; \mbox{for}  \; \; \; \; \;  \Delta t \leq \Delta t_{FE},
\end{eqnarray}
for the desired stability property indicated by the convex functional $\| \cdot \|$.  
For multiderivative methods, in addition to the first derivative, we need to appropriately 
approximate the  second derivative in time $u_{tt}$, to which we represent the discretization
 as $\dot{F}$. 
It is not immediately obvious what should be the form of a  condition
that would account for the effect of the $\Delta t^2  \dot{F}$ term. 
Motivated by the examples in the previous sections, we choose the 
\begin{eqnarray} \label{vanishing}
\mbox{\bf Second derivative condition} \; \; \; \; \; \; \; \;  
\|u^n +\Delta t^2 \dot{F}(u^n) \| \leq \| u^n  \| \; \; \; \; \; \mbox{for}  \; \; \; \; \;  \Delta t \leq K \Delta t_{FE},
\end{eqnarray}
where $K$ is a scaling factor that compares the stability condition of the second derivative term
to that of the forward Euler term.
Given conditions  \eqref{FE} and \eqref{vanishing}, we wish to formulate sufficient 
conditions so that the multiderivative  method \eqref{MSMD} satisfies the desired monotonicity condition
under a given time-step. First, we write the method 
\eqref{MSMD} in an equivalent matrix-vector form 
\begin{eqnarray} \label{MSMDvector}
\vy = \ve u^n + \Delta t S F(\vy) +  \Delta t^2 \hat{S} \dot{F}(\vy),
\end{eqnarray}
where 
\[ 
	S= \left[ \begin{array}{ll} A & \textbf{0} \\ b^T & 0 \end{array} \right] \; \; \; \; \;
	\mbox{and} \; \; \; \; \; \hat{S}= \left[ \begin{array}{ll} \hat{A} & \textbf{0} \\ \hat{b}^T & 0  \end{array}  \right]
\] and $\ve$ is a vector of ones.
We are now ready to state our result:

\begin{thm} Given spatial discretizations $F$ and $\dot{F}$ that satisfy \eqref{FE} and \eqref{vanishing},
a two-derivative multistage method  of the form \eqref{MSMDvector} preserves  the strong stability property 
$ \| u^{n+1} \| \leq \|u^n \|$  under the time-step restriction $\Delta t \leq r \DtFE$ if  satisfies the conditions
\begin{subequations} 
\begin{align} 
\left( I +r S  + \frac{r^2}{K^2} \hat{S} \right)^{-1}  \ve  \geq  0   \label{SSPcondition1} \\
r  \left( I +r S  + \frac{r^2}{K^2} \hat{S} \right)^{-1} S \geq  0   \label{SSPcondition2} \\
 \frac{r^2}{K^2}   \left( I +r S  + \frac{r^2}{K^2}  \hat{S} \right)^{-1} \hat{S} \geq 0   \label{SSPcondition3}
 \end{align} 
\end{subequations}
for some $r>0$. In the above conditions, 
the inequalities are understood component-wise. 
\end{thm}
\begin{proof} We begin with the method \eqref{MSMDvector}, and 
add the terms $r S \vy$ and $\hat{r} \hat{S} \vy$ to both sides to obtain
\begin{eqnarray*}
\left( I +r S  + \hat{r} \hat{S} \right)  \vy &=& u^n \ve + r S \left( \vy + \frac{\Delta t}{r}  F(\vy) \right)
 + \hat{r} \hat{S}  \left( \vy +  \frac{\Delta t^2}{\hat{r}} \dot{F}(\vy) \right),   \\
 \vy & = & R (\ve u^n) + P \left( \vy + \frac{\Delta t}{r}  F(\vy) \right) + Q  \left( \vy +  \frac{\Delta t^2}{\hat{r}} \dot{F}(\vy) \right),
\end{eqnarray*}
where \[
	R  =  \left( I +r S  + \hat{r} \hat{S} \right)^{-1}, \; \; \; \; \; \;
	P  =  r R S,  \; \; \; \; \; \;
	Q =   \hat{r} R  \hat{S}. 
\]
If the elements of $P$, $Q$, and $R \ve$ are all non-negative, and if $ R + P + Q = I$, then these 
 three terms describe a convex combination of terms which are SSP, and the resulting value is SSP as well
\[  \| \vy \| \leq  R \|\ve u^n\| + P \| \vy + \frac{\Delta t}{r}  F(\vy) \| + Q  \| \vy +  \frac{\Delta t^2}{\hat{r}} \dot{F}(\vy) \| ,\]
under the time-step restrictions $ \dt \leq r \DtFE$ and $ \dt \leq K \sqrt{\hat{r}}  \DtFE$.
As we observed  above,  the optimal time-step is given when these two are set equal,  so we require  $r=  K \sqrt{\hat{r}}$.
Conditions   \eqref{SSPcondition1}--\eqref{SSPcondition3} now ensure that $P \geq 0 $, $Q\geq 0$, and $R \ve \geq 0$  component-wise for 
$\hat{r} = \frac{r^2}{K^2}$, and our   method \eqref{MSMDvector} preserves the strong stability condition 
$\| u^{n+1} \| \leq \| u^n \|$ under the time-step restriction
$\Delta t \leq r \Delta t_{FE} $. 
\end{proof}

This theorem gives us the conditions for the method \eqref{MSMDvector} to be SSP for any  the time-step $\dt \leq r \DtFE$. This 
allows us to formulate the search for optimal SSP two-derivative methods as an optimization problem,
similar to \cite{ketcheson2008, Ketcheson2009, SSPbook2011}, 
where the aim is  to find $\sspcoef = \max r$ such that the  relevant order conditions (from Section \ref{MSMDoc}) and SSP conditions 
\eqref{SSPcondition1}-\eqref{SSPcondition2}
are all satisfied. 
 Based on this, we wrote a {\sc Matlab} optimization code for finding optimal two-derivative multistage methods \cite{MSMDoptcode}, 
formulated along the lines of  David Ketcheson's code \cite{ketchcodes} for finding optimal SSP multistage multistep methods in \cite{tsrk, msrk}. 
We used this  to find optimal SSP multistage two-derivative methods of order up to $p=5$. However, we also used our observations on the resulting
methods to formulate closed form representations of the optimal SSP multistage two-derivative methods. We present both
the numerical and closed-form optimal methods in the following section.


\begin{rmk}
An alternative approach to defining a MSMD SSP method is to begin with  spatial discretization
that satisfies the ``Taylor series'' condition, defined by
\begin{eqnarray} \label{TS}
\|u^n + \Delta t F(u^n) + \frac{1}{2} \Delta t^2 \dot{F}(u^n) \| \leq \| u^n  \| 
\; \; \; \; \; \mbox{for}  \; \; \; \; \;  \Delta t \leq \Delta t_{TS} = K_2  \Delta t_{FE}.
\end{eqnarray}
This condition replaces \eqref{vanishing}, and allows us to rewrite \eqref{MSMD} as convex combinations of 
forward Euler and Taylor series steps of the form \eqref{TS}. A similar optimization problem can be defined based
on this condition. This approach was adopted in \cite{Nguyen-Ba2010}.
\end{rmk}
This condition is more restrictive than what we consider in the present work.
Indeed, if a spatial discretization satisfies conditions
\eqref{FE} as well as \eqref{vanishing}, then it will also satisfy condition \eqref{TS}, with $K_2= K \sqrt{K^2 +2 } - K$.
However, some methods of interest
cannot be written using a Taylor series  decomposition,
including the two-stage fourth order method in \eqref{2s4p}.
For these reasons, we do not explore this approach further, but we point out that
we have experimented with this alternative formulation to generate some optimal SSP methods. 
Henceforth, we restrict our attention to spatial discretizations that satisfy \eqref{FE} and \eqref{vanishing}.

\section{Optimal SSP multiderivative methods}

\subsection{Second order methods}
Although we do not wish to use second order methods in our computations, it is interesting to consider the strong stability
properties of these methods both as building blocks of higher order methods and as simple example that admit
optimal formulations with simple formulas.

\noindent{\bf One stage methods.}  The second-order Taylor series method,
\begin{eqnarray} \label{TS2}
 u^{n+1}  =  u^n +  \Delta t F(u^n) + \frac{1}{2} \Delta t^2 \dot{F}(u^n),
 \end{eqnarray}
is the unique one-stage second-order method. 
 Theorem 1 
for the building block \eqref{BBex} with coefficients $\alpha=1$ and $\beta= \frac{1}{2}$ gives the condition
$\dt \leq \sspcoef  \DtFE$ with 
$ \sspcoef =  K \sqrt{K^2 + 2 } -  K^2.$ 
 Unlike the SSP single derivative
Runge--Kutta methods, the SSP coefficient is not just dependent on the time-stepping method, but also on the 
value $K$ which comes from the second derivative condition \eqref{vanishing}.
 As noted above,  the Taylor series method can serve as the basic building block for two-derivative methods,
just as the first-order FE can be used to build higher-order RK integrators.

\noindent{\bf Two stage methods.} 
The optimal SSP two stage second order methods depends on the value of $K$ in \eqref{vanishing}.
A straightforward SSP analysis via convex combinations, and solving the equations for the order conditions 
allows us to formulate the optimal methods without using the optimization code \cite{MSMDoptcode}. However, we used the 
optimization code to verify our results.

 If $K \leq  \sqrt{\frac{2}{3}}$  the optimal  methods are of the form
\begin{eqnarray} \label{2s2pKsmall}
u^* & = & u^n + \frac{1}{r} \Delta t F(u^n),  \nonumber \\
u^{n+1} & = & u^n + \frac{1}{2} \Delta t  \left( F(u^n) +  F(u^*) \right)
	+ \frac{r-1}{2r} \Delta t^2 \dot{F}(u^n),
\end{eqnarray}
where 
\[ r = \frac{1}{2} \left( 1 - K^2 + \sqrt{1 + 6 K^2 + K^4} \right) .\]
These methods are SSP for $\sspcoef=r$, where 
$r > 1$ whenever $K > 0$. Notice that if $K = 0$ we have  $r=1$ and our method
reduces to the standard two stage second order Runge--Kutta method.


As $K$ increases, the time-step restriction imposed by the condition \eqref{vanishing} is alleviated, and consequently the 
optimal method includes more of the second derivative terms. If
$K \geq \sqrt{\frac{2}{3}}$ we get an optimal method 
\begin{eqnarray} \label{2s2pKlarge}
    u^*     & = & u^n + \frac{1}{2} \Delta t F(u^n) + \frac{1}{8} \Delta t^2 \dot{F}(u^n), \nonumber \\
    u^{n+1} & = & u^* + \frac{1}{2} \Delta t F(u^*) + \frac{1}{8} \Delta t^2 \dot{F}(u^*),
\end{eqnarray}
which is simply two Taylor series steps with $\frac{1}{2} \Delta t$ in each one, and so
has $\sspcoef = 2 K \sqrt{K^2 + 2 } - 2 K^2$.
We see in Figure \ref{fig:2s2p} the SSP coefficient $\sspcoef$ vs. $K$ for the two methods above (in  green and red respectively),
and for many numerically optimal methods in blue.

\begin{figure}[!ht]
   \centering
    \begin{minipage}{.425\textwidth}
        \centering
        \includegraphics[width=1.1\textwidth]{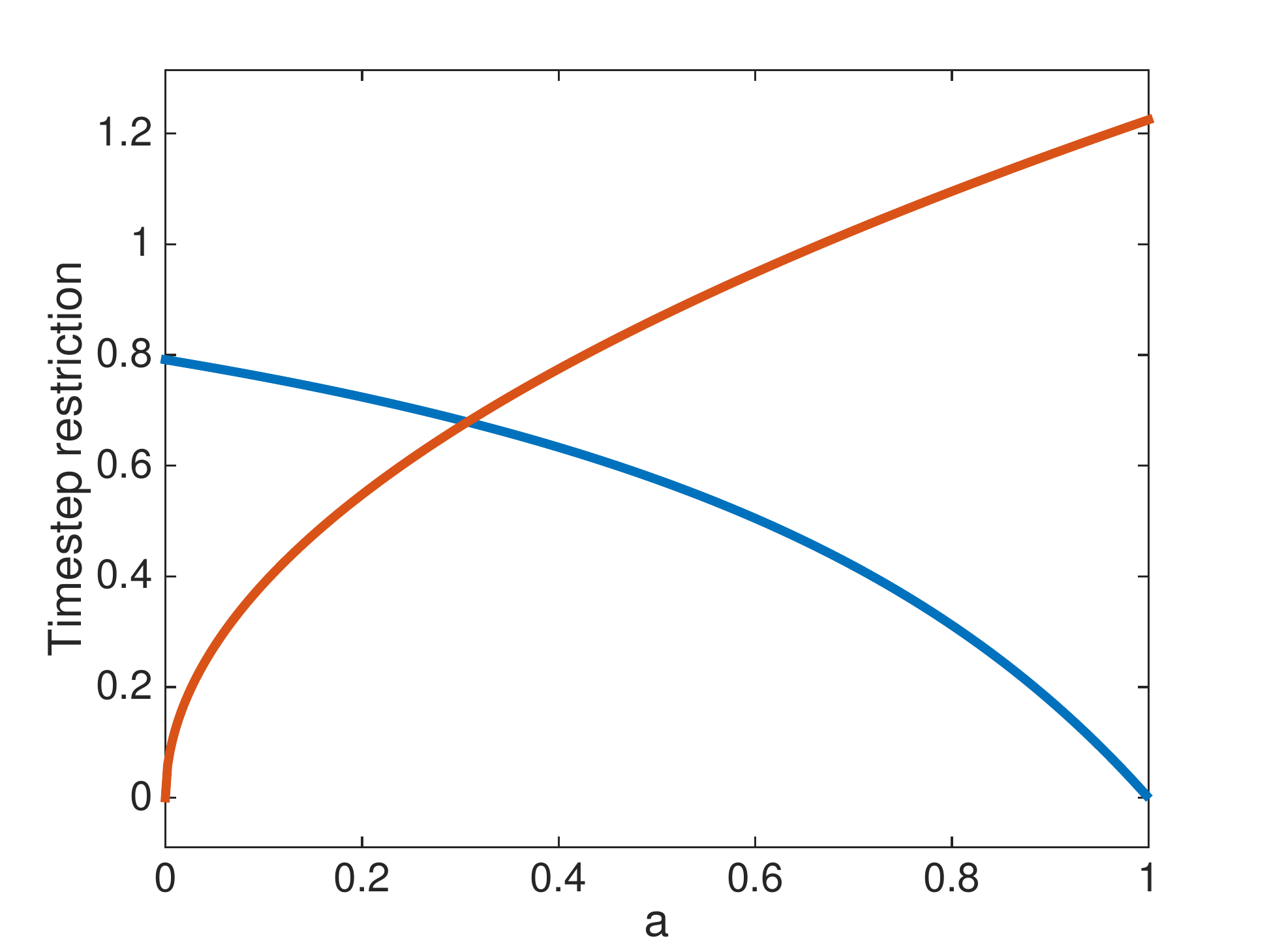}
\caption{The time-step restriction for the two-stage fourth order method is the minimum of two terms, one of which
is decreasing while the other is increasing in $a$.}
\label{Fig2s4pSSP} 
    \end{minipage}%
    \hspace{0.2in}
    \begin{minipage}{0.425\textwidth}
        \centering
        \includegraphics[width=1.1\linewidth]{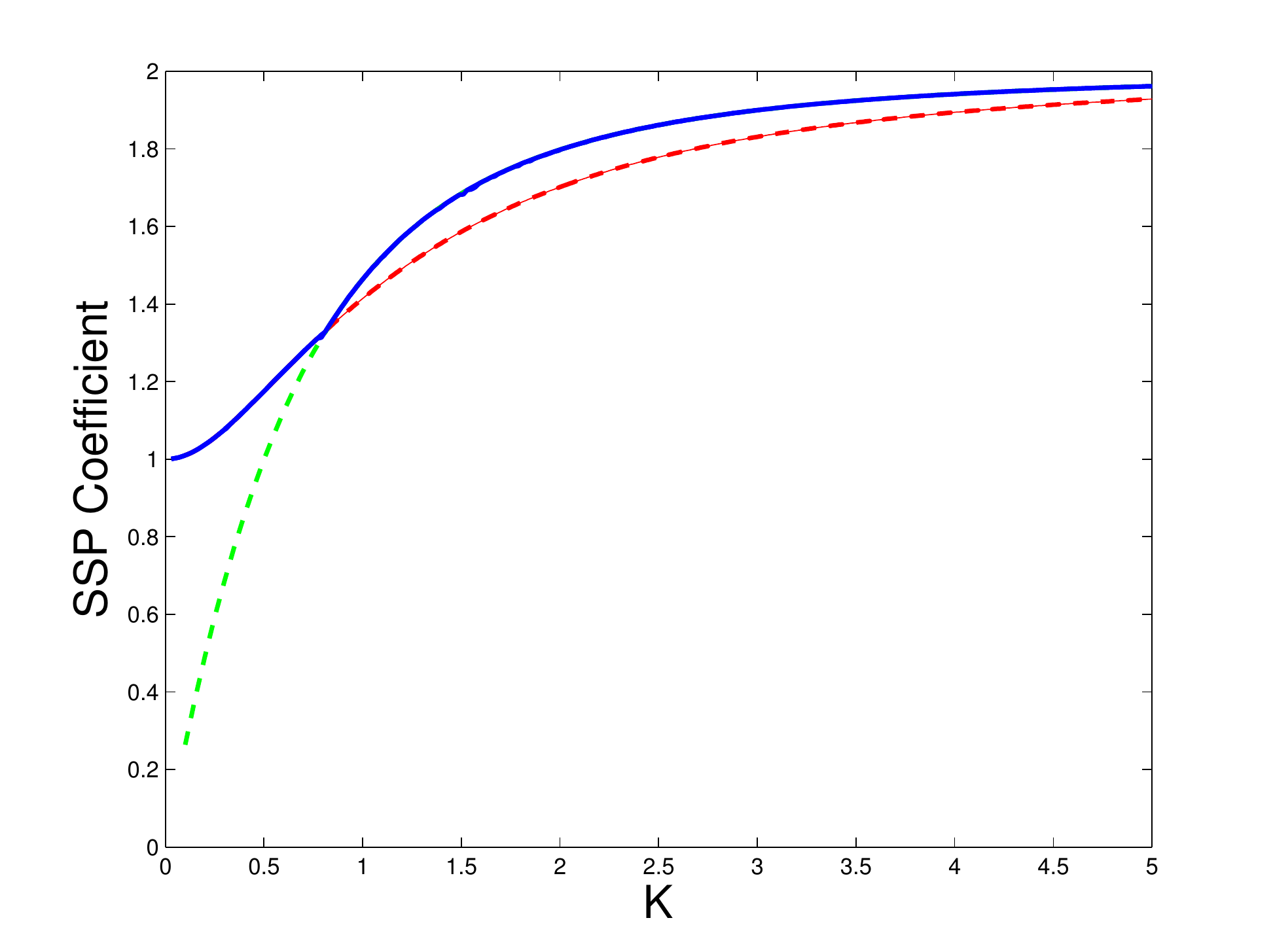}
\caption{The SSP coefficient $\sspcoef$ as a function of the coefficient $K$ in \eqref{vanishing}
for the optimal two-stage two-derivative second order methods given in Section 3.1. }
\label{fig:2s2p}
\end{minipage}
\end{figure}

Note that \eqref{2s2pKlarge} has four function evaluations compared to the three function 
evaluations above in \eqref{2s2pKsmall}. If we assume all the function evaluations cost the same, it will never pay off to use
the second method, as the first method is always more efficient. However, if one has a special case 
where the cost of computing $\dot{F}$ is negligible the second method may still be worthwhile.

For this method and all others, once we have the optimal Butcher arrays and the SSP coefficient $\sspcoef=r$, 
we can easily convert them to the Shu-Osher form as follows:
\begin{verbatim}
% Given A, Ahat, b, bhat, and r, the Shu-Osher matrices are given by
    z=0.0*b;  I=eye(3);e=ones(3,1);
    S=[A,z; b',0]; Shat=[Ahat,z; bhat',0];
    Ri=(I+r*S+(r^2/k^2)*Shat);    
    v=Ri\e;
    P = r*(Ri\S);
    Q= r^2/k^2*(Ri\Shat);
\end{verbatim}

\subsection{Third order methods}
Many  two-stage two-derivative third order methods exist. 
As above, the  optimal $\sspcoef$ depends on the value of $K$ in \eqref{vanishing}. 
Through numerical search using our optimization code \cite{MSMDoptcode}, we found optimal methods for the range $0.1 \leq K \leq 5$. 
Figure \ref{2stage} (blue line) shows the SSP coefficient $\sspcoef$ vs. $K$.
We also solved the  order conditions explicitly and analyzed the SSP conditions 
\eqref{SSPcondition1}-\eqref{SSPcondition3}
to find the values of the coefficients of these methods as functions of 
$r= \sspcoef $ and $K$. 
These methods all have the form
\begin{eqnarray} \label{2s3p}
u^*     &=& u^n+a \Delta t F(u^n)+ \hat{a} \Delta t^2 \dot{F}(u^n), \nonumber \\
u^{n+1} &=& u^n +b_1  \Delta t  F(u^n) +  b_2 \Delta t  F(u^*) +\hat{b_1}
    \Delta t^2  \dot{F}(u^n) + \hat{b_2} \Delta t^2  \dot{F}(u^*),
\end{eqnarray}
where the coefficients satisfy
\begin{eqnarray}
a & =  \frac{1}{r} \left(K \sqrt{K^2+2} - K^2 \right) \; \; \; \; \; & \hat{a}  =  \frac{1}{2} a^2 \nonumber \\ 
b_2 & =  \frac{2 K^2 (1 - \frac{1}{r} ) + r } {  K \sqrt{K^2+2} + K^2} - \frac{r^2}{3 K^2} \; \; \; \; \;
& b_1  =  1- b_2 \\
\hat{b}_1 & =  \frac{1}{2} - \frac{1}{2} a b_2 - \frac{1}{6 a} \; \; \; \; \;
&\hat{b}_2  =  \frac{1}{6 a} - \frac{1}{2} a b_2 \nonumber
\end{eqnarray}

Note that the first stage is always a Taylor series  method with $\Delta t$ replaced by $a \Delta t$, which is underscored by the fact that 
SSP coefficient is given by $r  =  \frac{1}{a} \left(K \sqrt{K^2+2} - K^2 \right)$. The SSP conditions
\eqref{SSPcondition1}-\eqref{SSPcondition3}
 provide the restrictions on the size
of $r$, and thus on $\sspcoef=r$. We observe that the condition that restricts us most here is the non-negativity of $Q_{3,1}$, 
and so we must select a value of $r$ such that this value is zero. To do this, we select $r$ to be the smallest positive root of 
$p_0 + p_1 r + p_2 r^2 + p_3 r^3 $
where the coefficients $p_i$ are all functions of $K$
\begin{eqnarray*}
 p_0  & = & 2 K ( \sqrt{K^2+2} - 3K) + 4 K^3 ( \sqrt{K^2+2} - K)  \\
    p_1&=& -a_0 ,\; \; \; \; \; p_2=(1-a_0)/(2 K^2), \; \; \; \; \;
    p_3= -\frac{\frac{a_0}{2K} + K}{6 K^3}. 
    \end{eqnarray*}
Some examples of the value of $r$ as a function of $K$ are given in Table \ref{table2s3p}. The {\sc Matlab} script for this method
is found in Appendix B.

 \begin{table}[!h] 
\hspace{0.65in} {\small \begin{tabular}{|lllllllllllllll|} \hline
K & 
0.25 & 0.4 & 0.5 & 0.6 & 0.7 & 0.8 & 1.0 & 1.25 & 1.5 & 1.75 & 2.5 & 3 & 3.5 & 4 \\
r & 0.48  & 0.71 & 0.84   &0.94  &1.03     & 1.11    &1.23      & 
1.33  &1.39 &1.44 & 1.51 & 1.54  &1.55  &1.56              \\  \hline
\end{tabular}}
\caption{The SSP coefficient $\sspcoef=r$ for each $K$ in the two-stage third order method \eqref{2s3p}.}
\label{table2s3p}
\end{table}

The Shu-Osher representations of the methods can be easily obtained by using $r$ and $K$ to solve for $\hat{r}$, $R, P$, and $Q$. 
For example, for $K  = \frac{1}{\sqrt{2}}$ we have $r=1.04$ and the method in Butcher form has
   \[ \begin{array}{l}
   a =  0.594223212099088,   \\
 \hat{a} = 0.176550612898679, \end{array}
\; \; \; \; \;
b = \left[ \begin{array}{l}
0.693972512991841  \\ 0.306027487008159  \end {array} \right] ,  \; \; 
 \hat{b} = \left[ \begin{array}{l}
 0.128597465450411 \\  0.189553898228989 \end {array} \right].
 \]
 The optimal Shu Osher formulation    for these values is given by $R \ve = (1,0,0)^T$,
     \[   P = \left[ \begin{array}{l l l}
  0&0&0  \\  0.618033988749895 &0&0 \\  0.271611333775367 & 0.318290138472780 & 0\\ \end {array} \right]\]
\[  Q = \left[ \begin{array}{l l l}
  0&0&0  \\   0.381966011250105 &0&0 \\ 0 &  0.410098527751853    & 0\\ \end {array} \right] .\]

\subsection{Fourth order methods}

\subsubsection{Fourth order methods: The two-stage fourth-order method}

The two-stage two-derivative fourth order method \eqref{2s4p} is unique; there is only one set of coefficients that satisfy the fourth
order conditions for this number of stages and derivatives. The method is
\begin{align*} 
    u^*     &= u^n + \frac{\Delta t}{2} F(u^n) + \frac{\Delta t^2}{8} \dot{F}(u^n) \nonumber \\
    u^{n+1} &= u^n + \Delta t           F(u^n) + \frac{\Delta t^2}{6}( \dot{F}(u^n)+2\dot{F}(u^*)).
\end{align*}
The first stage of the method is a Taylor series method with $\frac{\Delta t}{2}$, while the second stage can be written as
a linear combination of a forward Euler and a second-derivative term (but not a Taylor series term).
The SSP coefficient of this method is larger as $K$ increases, as can be seen in Figure \ref{2stage}.

\begin{figure}[t]
    \centering
    \begin{minipage}{.425\textwidth}
        \centering
        \includegraphics[width=.9\textwidth]{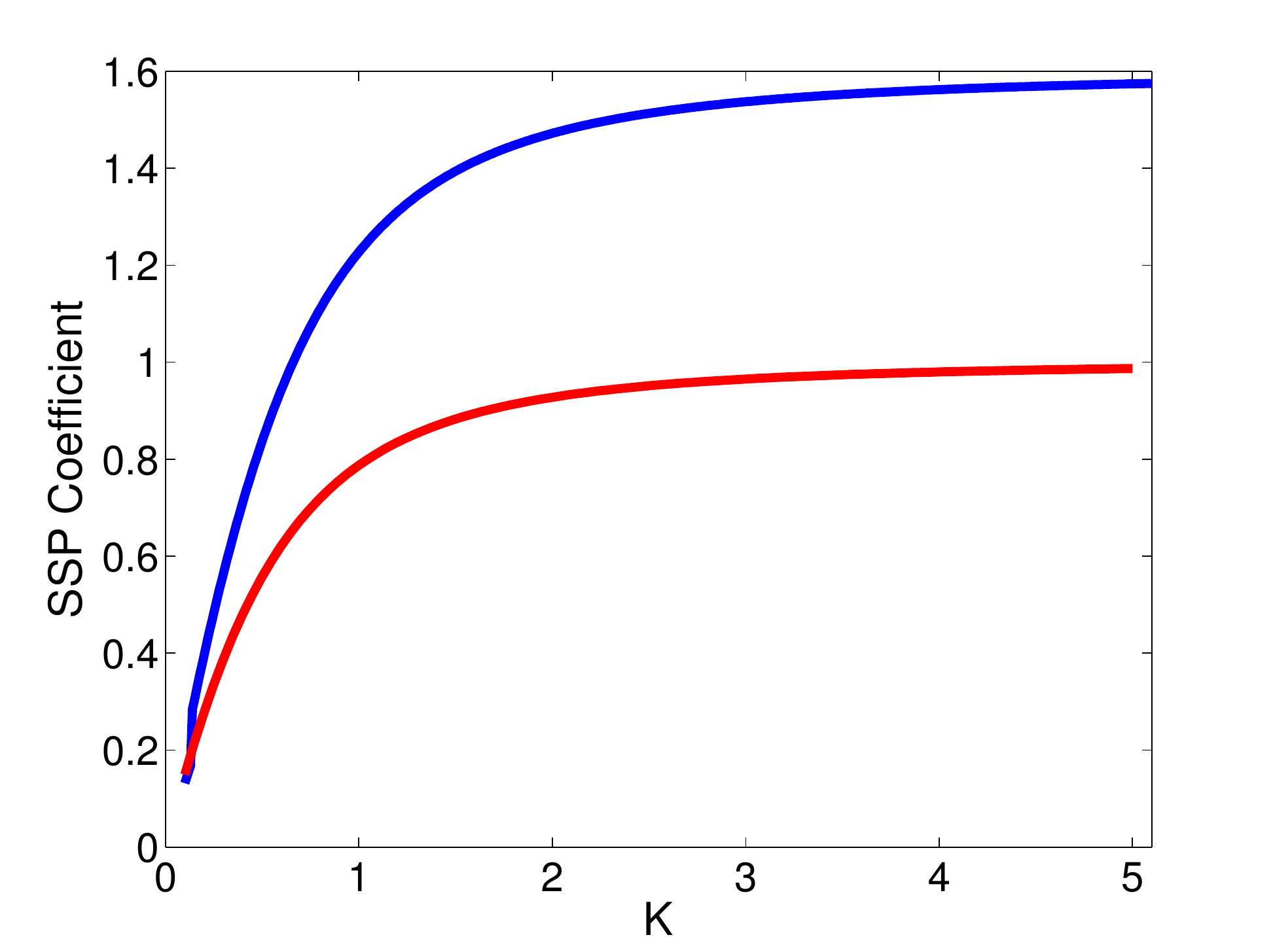}
       \caption{\small The SSP coefficient vs. $K$ for two stage methods. Third order methods are in blue, 
       the fourth order methods are red.}
        \label{2stage}
    \end{minipage}%
    \hspace{0.2in}
    \begin{minipage}{0.425\textwidth}
        \centering
        \includegraphics[width=0.9\linewidth]{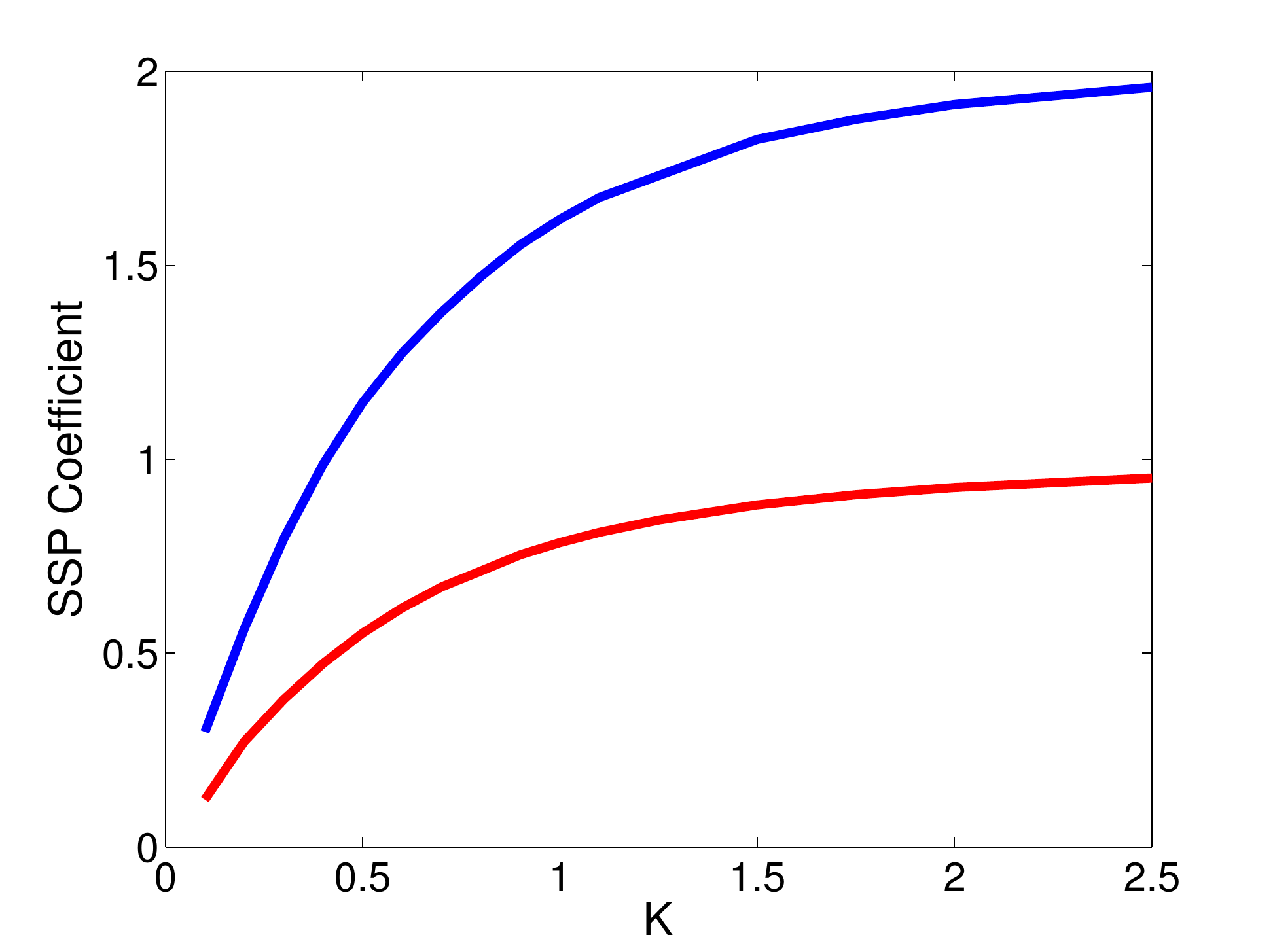}
        \caption{\small The SSP coefficient vs. $K$ for three stage methods. Fourth order methods are in blue (top line), 
        the fifth order methods in red (bottom line).}
        \label{3stage}
    \end{minipage}
\end{figure}

To ensure that the SSP conditions \eqref{SSPcondition1}-\eqref{SSPcondition3}
are satisfied, we need to select the largest $r$ so that 
all the terms are non-negative. We observed from numerical optimization that in the case 
of the two-stage fourth order method the term $(R \ve)_3$ gives the most restrictive condition:
if we choose $r$ to ensure that this term is non-negative, all the other conditions are satisfied.
Satisfying this condition, the SSP coefficient $\sspcoef=r$ is given by the smallest positive 
root of the polynomial:
\[ (R \ve)_3 = r^4 + 4 K^2 r^3 -12 K^2 r^2 - 24 K^4 r + 24 K^4.\] 
The Shu-Osher decomposition for the optimal method corresponding to this value of $K$ is
\begin{eqnarray}
u^* & = & \left( 1- \frac{4 r K^2 + r^2}{8 K^2} \right) u^n + \frac{r}{2} \left( u^n + \frac{\dt}{r} F(u^n) \right)
+ \frac{ r^2}{8 K^2}  \left( u^n + \frac{K^2}{r^2} \dt^2 \dot{F}(u^n) \right) \\
u^{n+1} &=& r \left( 1 - \frac{r^2}{6K^2} \right) \left( u^n + \frac{\dt}{r} F(u^n) \right) + 
	\frac{r^2 ( 4 K^2-r^2)}{24 K^4}  \left( u^n + \frac{K^2}{r^2} \dt^2 \dot{F}(u^n) \right) +
	\frac{r^2}{3 K^2}  				 \left( u^* + \frac{K^2}{r^2} \dt^2 \dot{F}(u^*) \right).  \nonumber
\end{eqnarray}

\subsubsection{Fourth order methods: Three  stage methods}
If we increase the number of stages to three, we can construct entire families of
methods that obtain fourth-order accuracy, and are SSP with a larger allowable time-step.
For these methods, we were not able to find closed form solutions, but our optimization code \cite{MSMDoptcode}
produced methods for various values of $K$. The SSP coefficient as a function of $K$ for these methods is given in
Figure \ref{2stage}, and we give the coefficients for selected methods in both Butcher array and Shu-Osher form 
in Appendix \ref{Appendix3s4pMethods}.

\subsection{Fifth order methods}
As mentioned above, it was shown that explicit SSP Runge--Kutta methods cannot have order 
$p>4$ \cite{kraaijevanger1991,ruuth2001}. This order barrier is broken by multiderivative methods. 
If we allow three stages and  two-derivative we can obtain a fifth order SSP method.
The explicit  three-stage fifth order method has twelve coefficients and sixteen order conditions that need to be satisfied.
This is possible if some of the coefficients are set to zero, which allows several of the order conditions to be repetitive and
satisfied automatically.  The methods resulting from our optimization routine all had the simplified form
\begin{eqnarray} \label{3s5p}
    u^*     &= & u^n +  a_{21} \Delta t F(u^n) +  \hat{a}_{21}\Delta t^2  \dot{F}(u^n) \nonumber \\
    u^{**}  &=& u^n +  a_{31} \Delta t F(u^n)  
    +  \hat{a}_{31}\Delta t^2  \dot{F}(u^n)  +  \hat{a}_{32}\Delta t^2  \dot{F}(u^*)  \\
    u^{n+1} &=& u^n +  \Delta t  F(u^n)  
    + \Delta t^2 \left( \hat{b}_1 \dot{F}(u^n)+ \hat{b}_2 \dot{F}(u^*) + \hat{b}_3 \dot{F}(u^{**}) \right). \nonumber
\end{eqnarray}
The coefficients of the three-stage fifth order method are then given as a one-parameter system, 
depending only on $a_{21}$, that are related through
\begin{eqnarray*}
\hat{a}_{21} &=& \frac{1}{2}  a_{21}^2 , \; \; \; \; \;   a_{31}    =   \frac{3/5 -a_{21}}{1-2 a_{21}}, \\ 
\hat{a}_{32}   & = & \frac{1}{10} \left(  \frac{(\frac{3}{5} -a_{21})^2}{a_{21} (1-2 a_{21})^3} - 
\frac{\frac{3}{5} -a_{21}}{(1-2 a_{21})^2}  \right), \; \; \; \; \; 
\hat{a}_{31}   =   \frac{1}{2} \frac{(\frac{3}{5} -a_{21})^2}{(1-2 a_{21})^2}  -\hat{a}_{32}, \\
\hat{b}_2  & = & \frac{2 a_{31}-1}{12 a_{21}(a_{31}-a_{21})} , \; \; \; \; \; 
\hat{b}_3   =  \frac{1-2 a_{21}}{12 a_{31} (a_{31}-a_{21})} , \; \; \; \; \; 
\hat{b}_1   =  \frac{1}{2} - \hat{b}_2 - \hat{b}_3. \\
\end{eqnarray*}

To satisfy the SSP conditions \eqref{SSPcondition1}-\eqref{SSPcondition3}, we must ensure that $(R\ve)_3$  is non-negative.  
Based on the optimization code we observed that the extreme case  of  $(R\ve)_3=0$ gives the optimal methods, 
 and we can obtain  $a_{21}$  as a function of $K$ and $r$ through
\[  a_{21} =\frac{K^6}{r^6} \left( - \frac{2}{K^4} r^5 + \frac{10}{K^4} r^4  +  \frac{40}{K^2} r^3
   - \frac{120}{K^2} r^2  -240 r + 240 \right).
\]
Now, we wish to ensure that $Q_{3,1}$ is nonnegative. The SSP coefficient $\sspcoef = r$ is then chosen 
   as the  largest positive root of
\[ Q_{3,1} = 10 r^2 a_{21}^4  - (100 K^2 + 10 r^2) a_{21}^3   +( 130 K^2  + 3 r^2 ) a_{21}^2   
- 50 K^2 a_{21}+ 6 K^2. \]
The {\sc Matlab} script in Appendix C solves for the largest $r$ that satisfies the SSP 
conditions \eqref{SSPcondition1}-\eqref{SSPcondition3}, and then computes the coefficients
of the optimal methods both in Butcher array and Shu-Osher form.
This approach yields the same optimal methods as those
obtained by our optimization code \cite{MSMDoptcode}. In Figure \ref{Table3s5p} we show values of $a_{21}$ and $r$ for given values of $K$.

\begin{figure}
 \begin{minipage}[b]{0.5\linewidth} 
    \centering 
   \includegraphics[width=.875\textwidth]{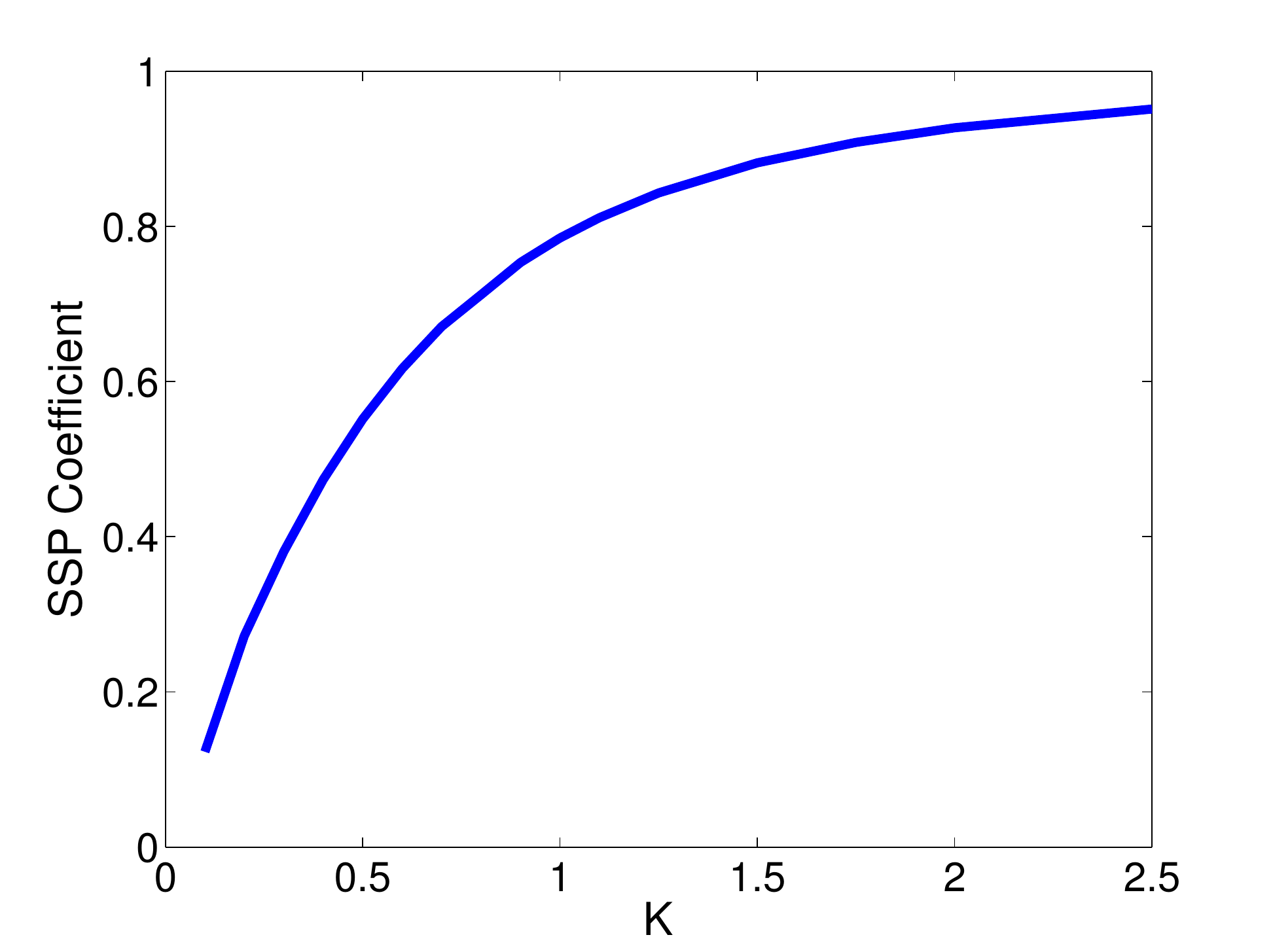}
   \par\vspace{0pt}
  \end{minipage}%
  \begin{minipage}[b]{0.5\linewidth} \hspace{-.32in}
    \centering  \vspace{-0.2in}
    \begin{tabular}{|lll||lll|} \hline
$K$ & $a_{21}$ & $\sspcoef$ & $K$ & $a_{21}$ & $\sspcoef$ \\  
\hline 
    0.1  &    0.7947   & 0.1452   &  1.1  &    0.7393   & 0.8114 \\
    0.2  &    0.7842    & 0.2722   & 1.2  &    0.7374    &0.8335\\
    0.3  &    0.7751    &0.3814   & 1.3  &    0.7359    &0.8523\\
    0.4  &    0.7674    &0.4741   & 1.4  &    0.7346   & 0.8683\\
    0.5  &    0.7609    &0.5520   & 1.5  &    0.7334    & 0.8819\\
    0.6  &    0.7555    &0.6171   & 1.6  &    0.7324    & 0.8937\\
    0.7  &    0.7510    &0.6712   & 1.7  &    0.7316    & 0.9039\\
    0.8  &    0.7472    &0.7162   &   1.8  &    0.7309   & 0.9127\\
    0.9  &    0.7441    &0.7537   &   1.9  &    0.7302   & 0.9205\\
    1.0  &    0.7415    &0.7851   &  2.0  &   0.7296    & 0.9273\\ \hline
    \end{tabular} \vspace{0.235in}
\par\vspace{12pt}
\end{minipage}
\caption{\small SSP coefficients for three-stage fifth-order methods.  Left: the SSP coefficient as a function of $K$ for three stage fifth order methods.
Right: a table of $a_{21}$ and the SSP coefficient $\sspcoef$, for different values of $K$ as defined in \eqref{3s5p}. The code
to generate the coefficients in Butcher and Shu-Osher form is given in the appendix.
}
\label{Table3s5p}
\end{figure}

Once again, the Shu-Osher decomposition is needed for the method to be SSP, and is easily obtained. 
For example, for $K  = \frac{1}{\sqrt{2}}$ we have $r=0.6747$ and the method becomes 
\begin{eqnarray} \label{3s5pShuOsher}
R \ve &=& \left[ \begin{array}{l}
1.0 \\
0.2369970626512336 \\
0.7810723816004148 \\
0.0 \\
\end {array} \right], \;
  P  =   \left[ \begin{array}{l l l l}   0&0&0&0  \\ 
0.5064804704259125 & 0& 0&0 \\
0.1862033791874200 & 0   & 0& 0\\ 
 0.5769733539128722  &0  &0 &  0\\
 \end {array} \right]  \nonumber  \\ 
     \end{eqnarray}
 \begin{eqnarray} 
   Q & = & \left[ \begin{array}{l l l l} 
   0&0&0&0  \\ 
0.2565224669228537 & 0& 0&0 \\ 
0& 0.0327242392121651 & 0& 0\\
0.0615083849004797 & 0.0803574544380432 & 0.2811608067486047 & 0 \nonumber\\
 \end {array} \right].  
    \end{eqnarray}
  
 These coefficients as well as the  coefficients for the optimal method for any value of $K$ can be  easily 
 obtained to high precision by the {\sc Matlab} code in Appendix C.


\section{Numerical Experiments}
\subsection{Numerical verification of the SSP properties of these methods}
\subsubsection{Example 1: Linear advection with first order TVD spatial discretization}
\label{subsec:linear-advection-first-order}

As a first test case, we consider  linear advection, $U_t - U_x =0$, 
with a first order finite difference for the first derivative and 
a second order centered difference for the second derivative defined in \eqref{eqn:low-order}
\[
	F(u^n)_j := \frac{u^n_{j+1}-u^n_j}{\Delta x} \approx U_x( x_j ), \; \; \; \; \; \;
		\mbox{and} \; \; \; \; \; \;
	\dot{F}(u^n)_j := \frac{u^n_{j+1}- 2 u^n_j + u^n_{j-1}}{\Delta x^2} \approx  U_{xx}( x_j ).
\]
Recall from \eqref{eqn:tvd-properties} that this first-order spatial discretization satisfies the
\begin{description}
\item{\bf Forward Euler condition}
$ u^{n+1}_j = u^n_j + \frac{\Delta t}{\Delta x} \left( u^n_{j+1} - u^n_j \right)$
is  TVD for $ \Delta t \leq \Delta t_{FE} = \Delta x$, and the
\item{\bf Second Derivative condition}
$ u^{n+1}_j = u^n_j + \left( \frac{\dt}{\dx} \right)^2 \left( u^n_{j+1} - 2 u^n_j + u^n_{j-1} \right) $
is TVD for $\Delta t \leq \frac{1}{\sqrt{2}} \Delta t_{FE} $.
\end{description}
Fo initial conditions, we use a step function \begin{equation}
\label{eqn:sq_wave_ic}
	u_0(x) = \left\{ \begin{array}{ll}
		1 & \text{if}\ \frac{1}{4} \leq x \leq \frac{1}{2}, \\
		0 & \text{otherwise},
\end{array} \right. 
\end{equation}
with a domain $x \in [0,1]$ and periodic boundary conditions. 
This simple example is chosen as our experience has shown \cite{SSPbook2011} that this problem often 
demonstrates the sharpness of the SSP time-step.

 For all of our simulations, we use a fixed grid of size $\Delta x = \frac{1}{1600}$, and a time-step $\Delta t = \lambda \dx$
 where we vary $0.05 \leq \lambda \leq 1$. We step each method forward by $N=50$ time-steps and compare the 
 performance of the various time-stepping methods  constructed earlier in this work, for $K = \frac{\sqrt{2}}{2}$. 
 We test this problem using  the  two stage third order  method \eqref{2s3p}, 
the two stage fourth order method \eqref{2s4p} and the three stage fourth order method in Appendix \ref{Appendix3s4pMethods}, 
the fifth order method \eqref{3s5pShuOsher}. We also consider the non-SSP two stage third order method, 
 \begin{eqnarray} \label{bad2s3p}
u^*     &=& u^n - \Delta t F(u^n)+ \frac{1}{2} \Delta t^2 \dot{F}(u^n), \nonumber \\
u^{n+1} &=& u^n  - \frac{1}{3}  \Delta t  F(u^n) +  \frac{4}{3} \Delta t  F(u^*) +
\frac{4}{3}    \Delta t^2  \dot{F}(u^n) +\frac{1}{2} \Delta t^2  \dot{F}(u^*),
\end{eqnarray}
To  measure the effectiveness of these methods, we consider the maximum observed rise in total variation, defined by
\begin{equation}
\label{eqn:increase-in-tv}
	\max_{0 \leq n \leq N-1} \left( \| u^{n+1} \|_{TV} - \| u^n \|_{TV} \right). 
\end{equation}
We are interested in the time-step in which this rise becomes evident (i.e. well above roundoff error).
Another measure that we use is the rise in total variation compared to the total variation of the initial solution:
\begin{equation}
\label{eqn:tv-vs-inital}
	\max_{0 \leq n \leq N-1} \left( \| u^{n+1} \|_{TV} - \| u^0 \|_{TV} \right). 
\end{equation}

\begin{figure}[t!] \hspace{0.25in} 
\includegraphics[width=0.45\textwidth]{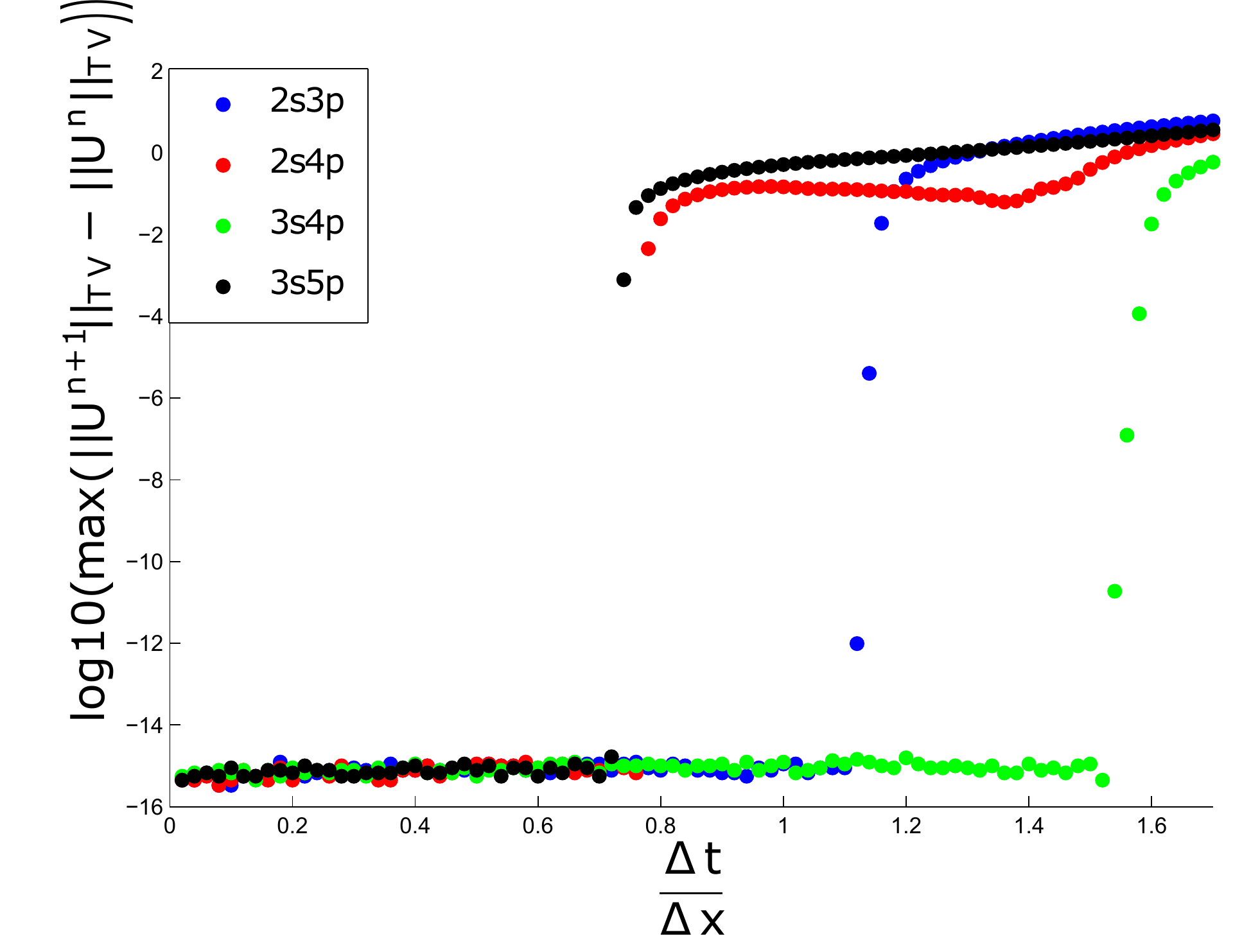}
\includegraphics[width=0.45\textwidth]{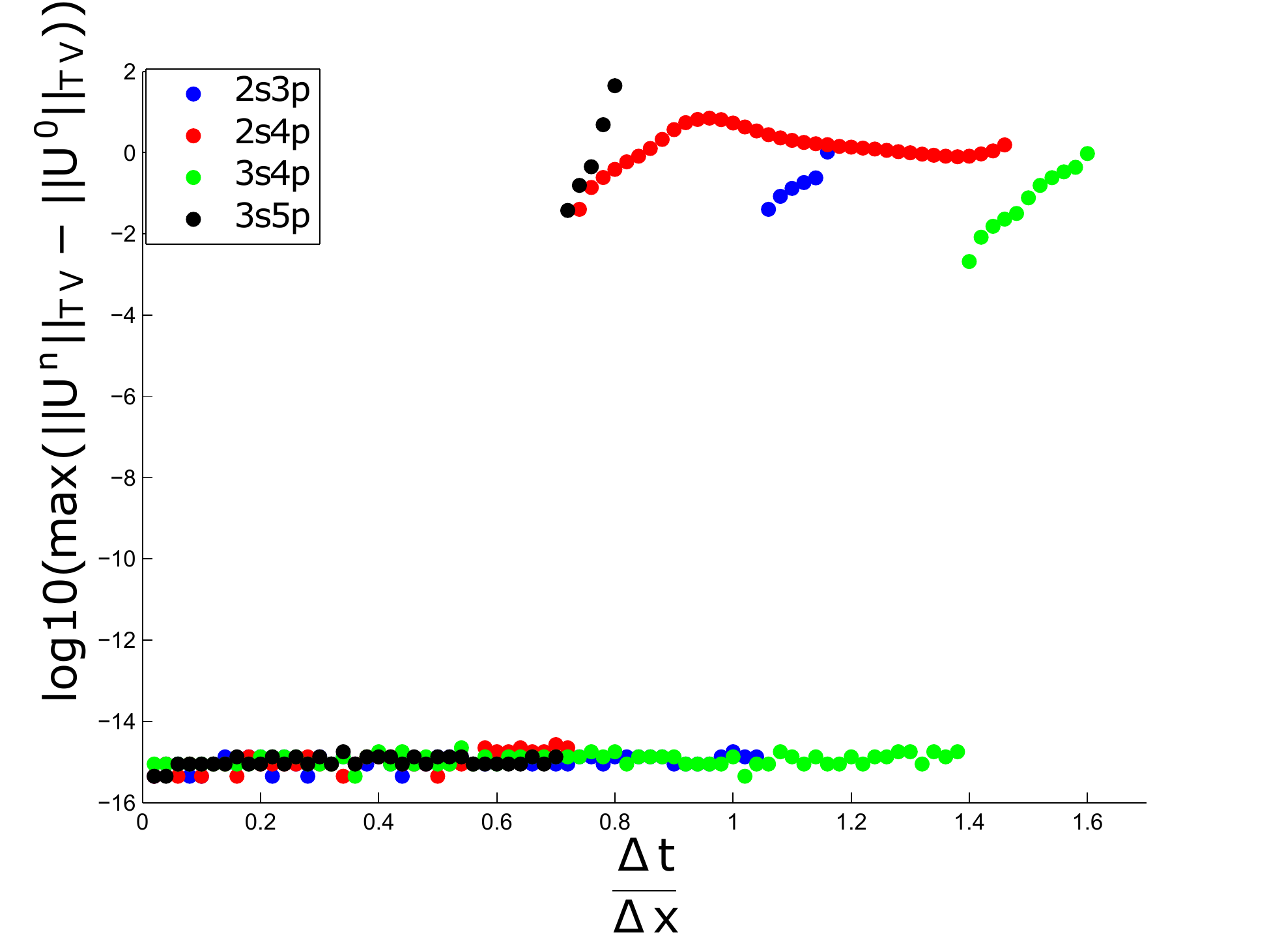} \\
\hspace*{.325in} \includegraphics[width=0.45\textwidth]{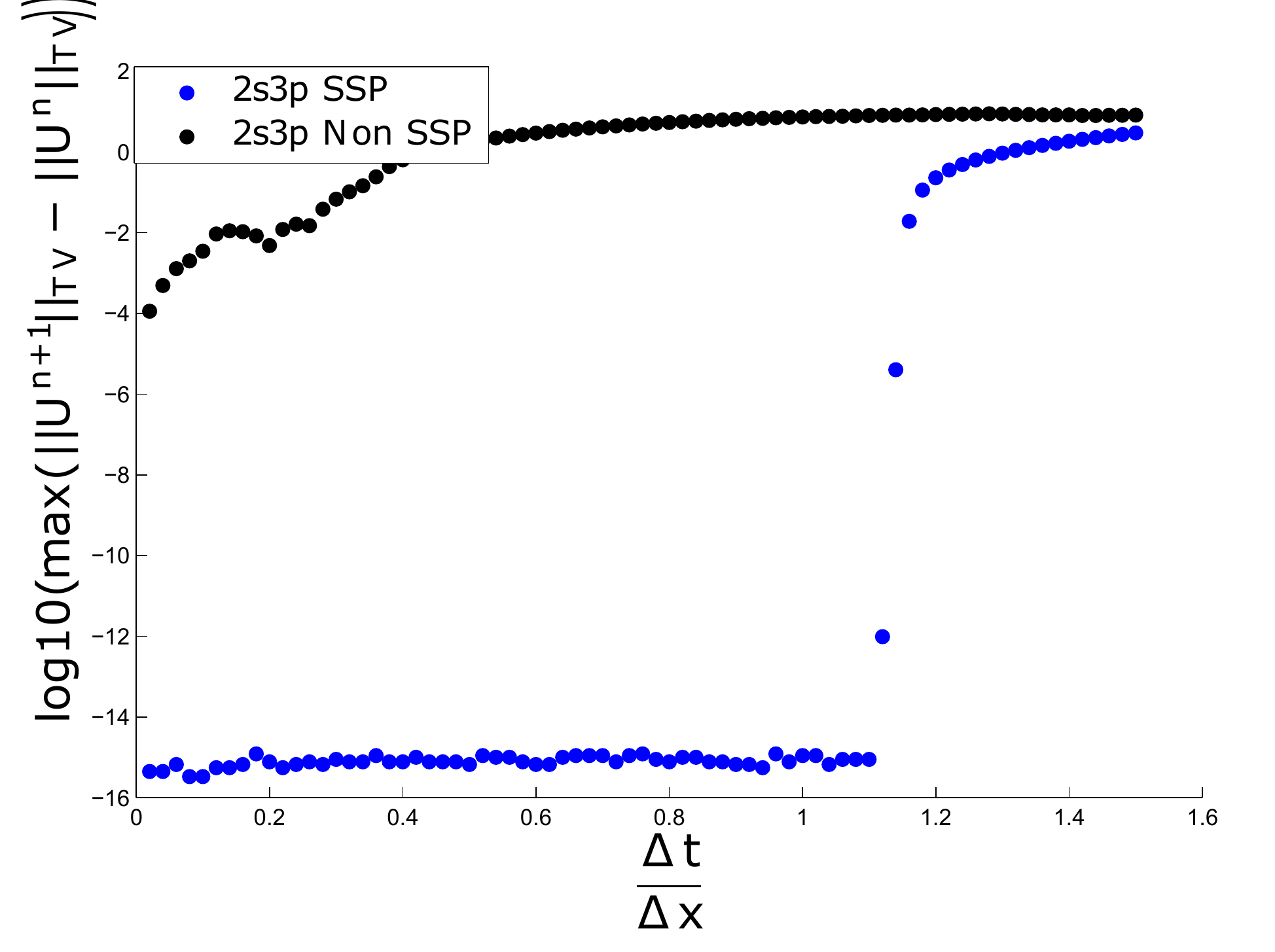}
\includegraphics[width=0.45\textwidth]{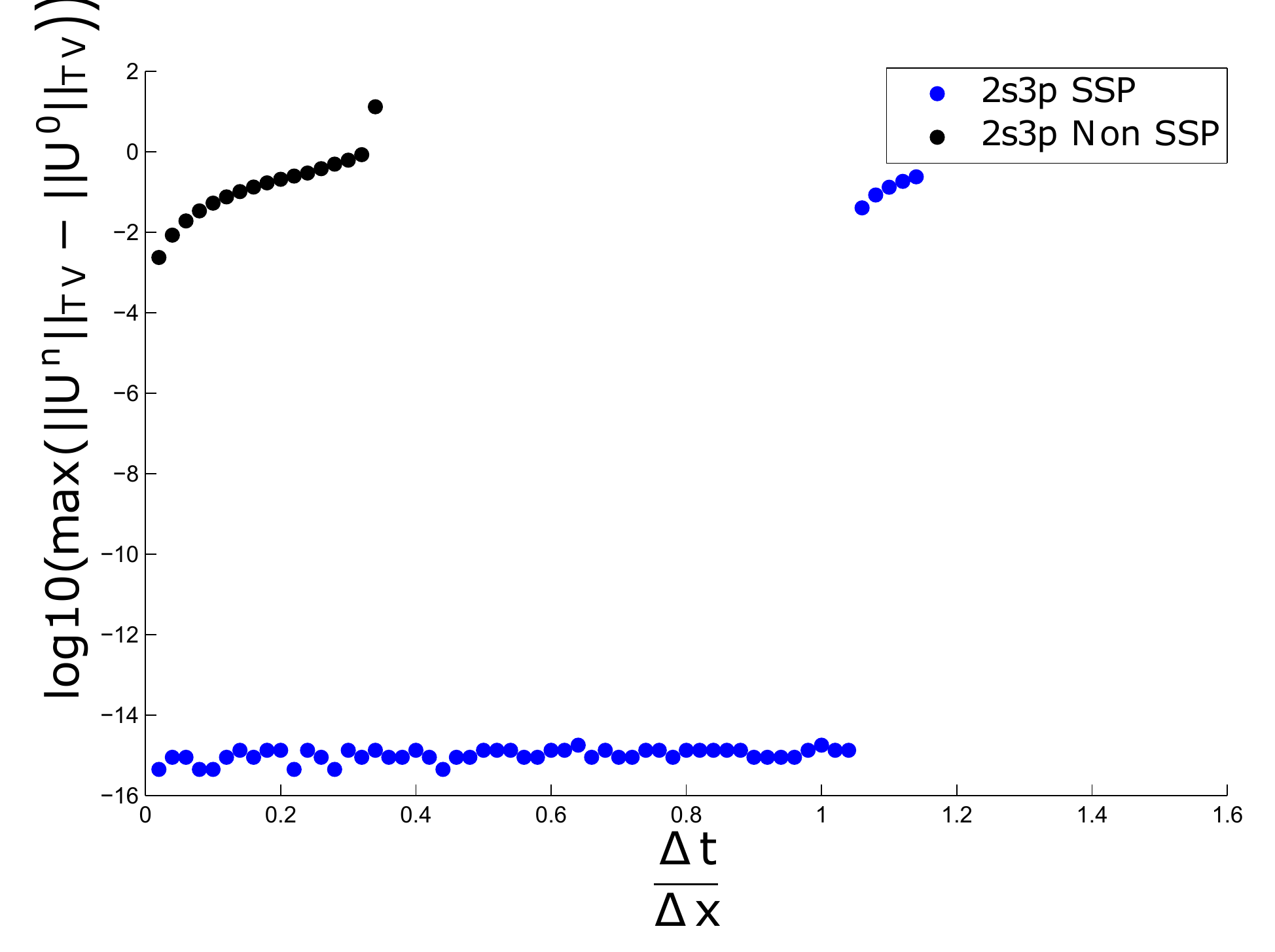}
\caption{The rise in total variation as a function of the CFL number.
On the left is the maximal per time-step rise \eqref{eqn:increase-in-tv} and on the right the maximal TV rise above
the initial TV \eqref{eqn:tv-vs-inital}.
Top: Comparison of a variety of SSP methods.
Bottom: Comparison of two-stage third order SSP and non-SSP methods.
 }
 \label{LinearTest}
\end{figure}

Figure \ref{LinearTest} (top) shows the maximal rise in total variation for each CFL value $\frac{\dt}{\dx}$. On the left we 
have the maximal per-step rise in TV \eqref{eqn:increase-in-tv} and on the right, the maximal rise in TV 
compared to the TV of the initial solution  \eqref{eqn:tv-vs-inital}. We clearly see that once the CFL value
passes a certain limit, there is a sharp jump in the total variation of the solution. We are interested in the value 
of $\frac{\dt}{\dx}$ at which the time-stepping  method no longer maintains the nonlinear stability. 
The fifth order method (black) has the most restrictive value of $\dt$ before the total variation begins to rise.
The next most restrictive is the two-stage fourth order method, followed by the two stage third order method. The two-stage
second order methods have more freedom than these methods, and so have a much larger allowable $\dt$, while the 
three stage fourth order, having the most freedom in  the choice of coefficients, outperforms all the other methods. 
Figure \ref{LinearTest} (bottom) compares the performance of the non-SSP method to the SSP method. This graph clearly shows 
the need for the SSP property of the time-stepping, as the absence of this property results in the loss of the TVD
property for any time-step.

\begin{table}[t] 
\hspace{0.8in} \begin{tabular}{|c|c|c|c||c|c|c|c|}  \hline 
Stages & Order & Predicted $\sspcoef$ & Observed $\sspcoef$ 
& Stages & Order & Predicted $\sspcoef$ & Observed $\sspcoef$\\
\hline
   1 &   2 &   0.6180 &  0.6180 &   2 &   4 &  0.6788&  0.7320\\
   2 &   2 &   1.2807&   1.2807  &   3 &   4 &   1.3927&   1.3927\\
   2 &   3 & 1.0400& 1.0400  &  3 &   5 & 0.6746&   0.7136\\
   \hline
\end{tabular} 
\caption{Comparison of the theoretical and observed SSP coefficients that preserve the nonlinear stability properties in Example 1.}
\label{LinearDT}
\end{table} 

We notice that controlling the maximal rise of the total variation compared to the initial condition  \eqref{eqn:tv-vs-inital} requires a 
smaller allowable time-step, so we use this condition as our criterion for maximal allowable time-step.
A comparison of the  predicted  (i.e. theoretical) values of the SSP coefficient and the observed value 
for the Taylor series method, the two stage methods of order $p=2,3,4$, and the three-stage fourth and fifth order methods
are shown in Table \ref{LinearDT} 
We note that for the Taylor series method, the two-stage second order method, the two stage third order method, and
three stage fourth order method, the observed SSP coefficient matches exactly the theoretical value. On the other hand,
the two-stage fourth order and the three-stage fifth order, both of which have the smallest SSP coefficients (both in theory 
and practice), have a larger observed SSP coefficient than predicted. For the two-stage fourth order case this is expected, as 
we noted in Section 2.1 that the TVD time-step for this particular case is (as we observe here)  $(\sqrt{3}-1)$, larger than the 
more general SSP timestep $\sspcoef = 0.6788$.

\subsubsection{Example 2: MSMD methods with weighted essentially non-oscillatory (WENO) methods}
The major use of MSMD time-stepping would be in conjunction with high order methods for problems with shocks.
In this section we consider two scalar problems: the linear advection equation
\begin{eqnarray} \label{linadv}
U_t + U_x =0
\end{eqnarray}
and the nonlinear Burgers' equation
\begin{eqnarray} \label{burgers}
U_t + \left( \frac{1}{2} U^2 \right)_x =0
\end{eqnarray}
on $x \in (-1,1)$.
In both cases we use the step function  initial conditions \eqref{eqn:sq_wave_ic},
and periodic boundaries. We use $N=201$ points in the domain, so that $\Delta x =\frac{1}{100}$.

We follow our previous work \cite{sealMSMD2014} with a minor modification for the spatial discretization.
The spatial discretization is performed as follows: at each iteration we take the known value $u^n$ and compute the flux
$f(u^n) = u^n$ in the linear case and $f(u^n) = \frac{1}{2} \left(u^n \right)^2$ for Burgers' equation. Now to compute the 
spatial derivative $f(u^n)_x$ we use the WENO method \cite{jiang1996}. In our test cases, we can avoid flux splitting, as $f'(u)$
is strictly non-negative (below, we refer to the WENO method on a  flux with $f'(u) \geq 0$ as WENO$^+$ and to to the corresponding method
on a  flux with $f'(u) \leq 0$ as WENO$^-$). 

\begin{figure}[t!] \hspace{0.25in} 
%
\includegraphics[width=0.45\textwidth]{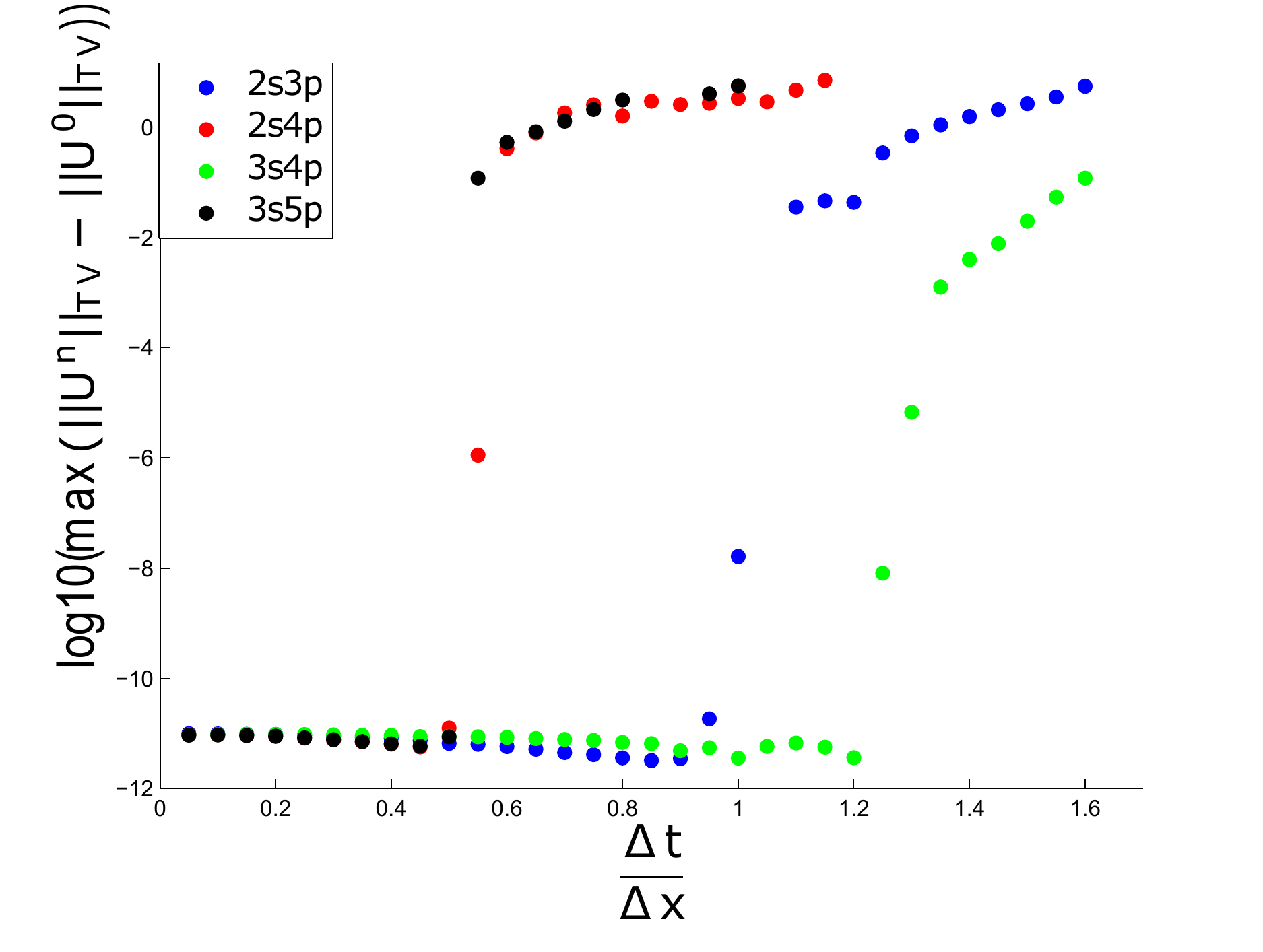}
\includegraphics[width=0.45\textwidth]{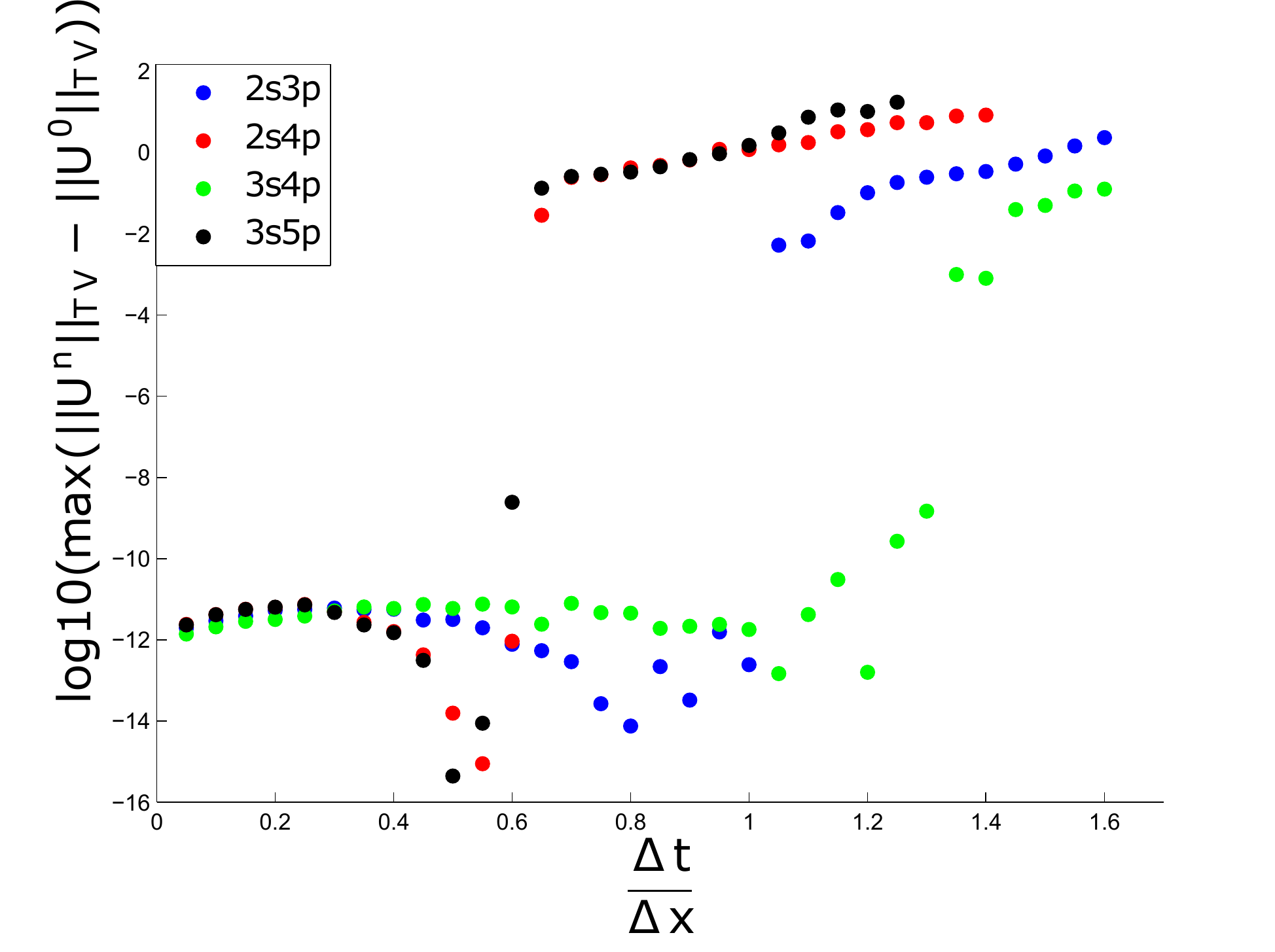} \\
\hspace*{0.3in} \includegraphics[width=0.45\textwidth]{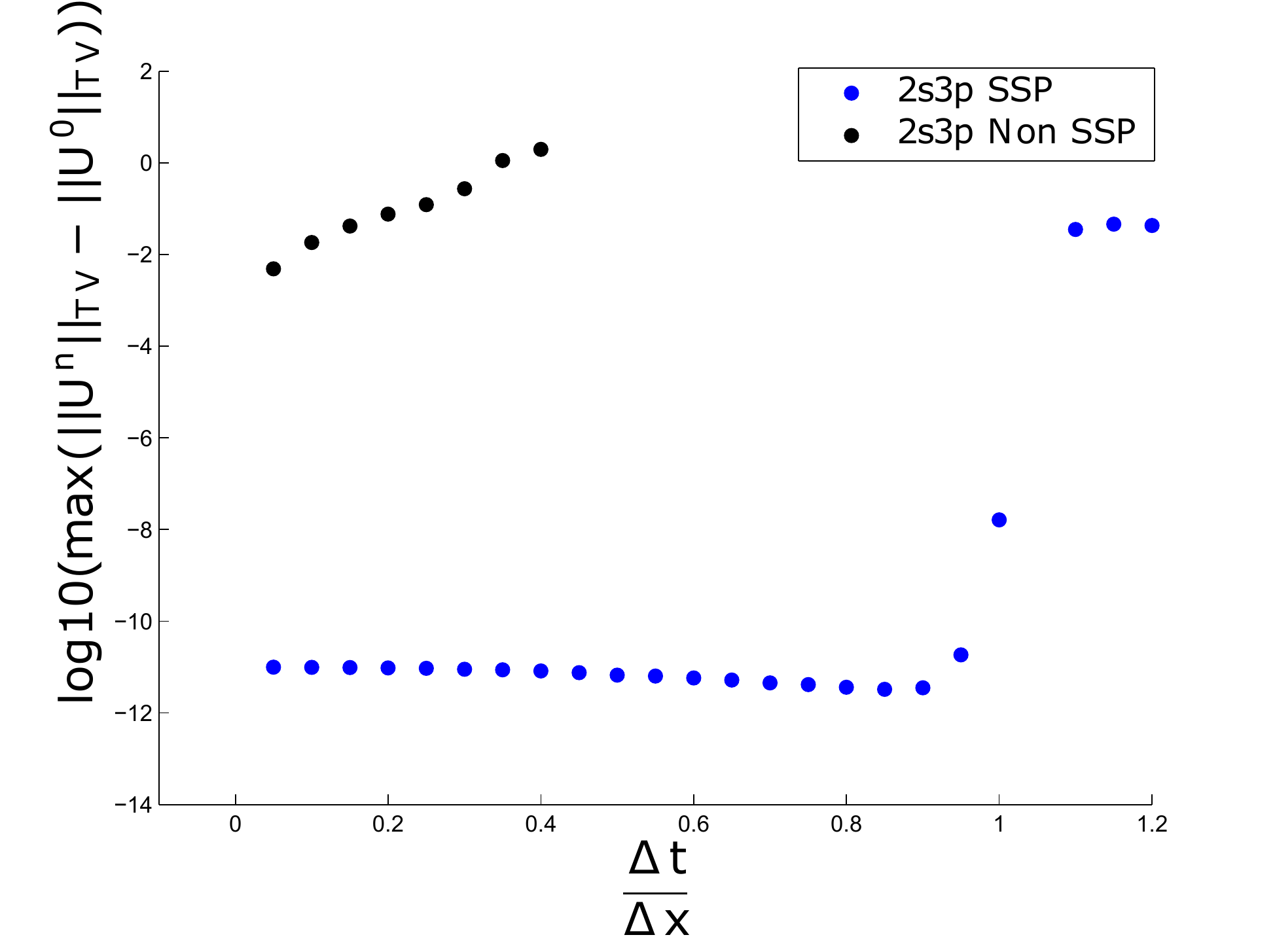}
\includegraphics[width=0.45\textwidth]{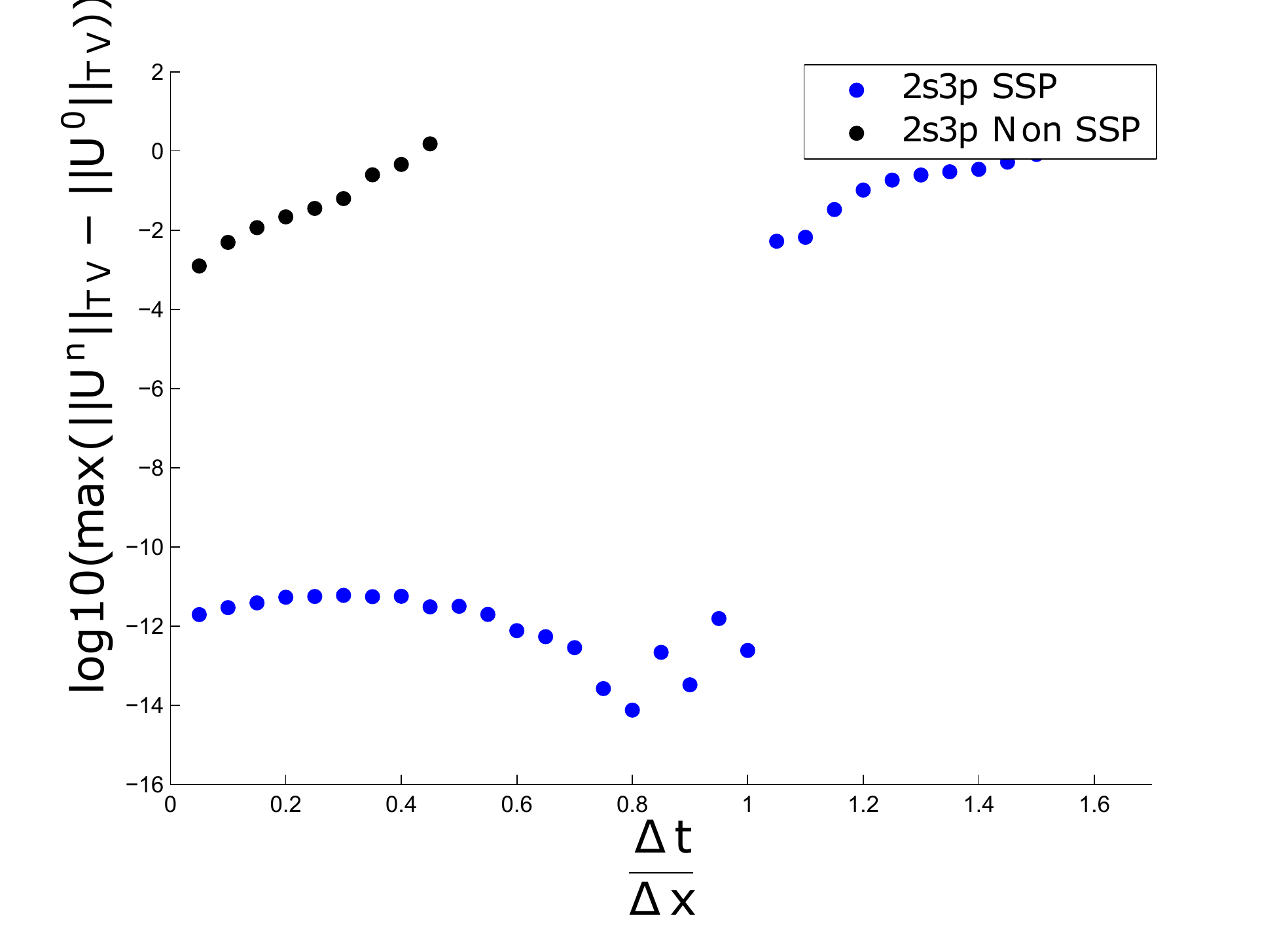}
\caption{Comparison of the rise in total variation as a function of the CFL number for Example 2. Linear advection on left and 
Burgers' equation on right. The top graphs compare the performance of different SSP methods, while the bottom
graphs compare the  two-stage third order SSP and non-SSP methods.
 }
 \label{WENOTest}
\end{figure}

Now we have the approximation to $U_t$ at time $t^n$, and wish to compute the approximation to $U_{tt}$. 
In previous work we defined the higher order derivative using central differences, but we have  found that 
additional limiting, in the form of the WENO$^-$ differentiation operator, is needed to achieve a pseudo-TVD like property.
For the linear flux, this is very straightforward as $U_{tt} =  U_{xx}$. To compute this, we take $u_x$ as computed before, 
and differentiate it using the WENO$^-$ method. Now we can compute the building block method.

For Burgers' equation, we have $U_{tt} = - \left( U U_t \right)_x$.  We take the approximation to $U_t$ that we 
obtained above using $WENO^+$, we multiply it by $u^n$ and differentiate in space using $WENO^-$. The choice of $WENO^{+}$ 
followed by $WENO^{-}$ is made by analogy to the first order finite difference  for the linear advection case, 
where we use a differentiation operator $D^{+}$  followed by the downwind differentiation operator $D^{-}$ to produce a centered difference
for the second derivative. The second derivative condition \eqref{vanishing} was satisfied by this approach.
Now we compute the building block
method with the approximations to $U_t$ and $U_{tt}$. In pseudocode, the building block calculation takes the form:
\begin{eqnarray*}
f(u^n)&=& \frac{1}{2}(u^n)^2;     \; \; \; \;  u^n_t =WENO^+(f(u^n));   \\
f'(u^n)&=&u^n ;    \; \; \; \; f(u^n)_t =f'(u^n)  u^n_t            \; \; \; \;
u^n_{tt} = WENO^-(f(u^n)_t) \\
u^{n+1} &=& u^n + \alpha \dt u^n_t +  \beta \dt^2 u^n_{tt} .
\end{eqnarray*}
We use the two stage   third order SSP method \eqref{2s3p}, and the non-SSP method \eqref{bad2s3p}
the two stage fourth order method \eqref{2s4p} and the three stage fourth order method in Appendix \ref{Appendix3s4pMethods}, 
the fifth order method \eqref{3s5pShuOsher}.  In these simulations, we use $\dt = \lambda \dx$ where $0.05 \leq \lambda \leq 1.6$, 
 and step up to $T_{final} =1.0$. At each time-step we compute  \eqref{eqn:tv-vs-inital},
 the maximal rise in total variation compared to the total variation of the initial solution. 
In Figure \ref{WENOTest} we  observe similar behavior to those of the linear advection with first order time-stepping, and once
again see that the SSP method is needed to preserve the nonlinear stability of WENO as well.

\subsection{Convergence studies} \label{convergence} 
As a final test case, we investigate the accuracy of the proposed schemes in conjunction with various high order spatial discretization operators.
 We perform several tests that demonstrate
that these methods converge with the correct order for linear and nonlinear problems.
In the first study (Example 3), we refine the grid only in time, and show that {\em if the spatial discretization is
sufficiently accurate} the  multi-derivative methods exhibit the design-order of convergence. 
We also compare the performance of the third order multi-derivative method to the three stage third order
explicit SSP Runge--Kutta method (SSPRK3,3) \cite{shu1988b} and show that the convergence properties are the same,
indicating that the additional error in approximating $F_t$ does not affect the accuracy of the method.
In the second study (Example 4) we co-refine the spatial and temporal grid by setting $\Delta t = \lambda \Delta x$
for a fixed $\lambda$, and shrink $\Delta x$.  We observed that since the order of the spatial method
is higher than the order of the temporal discretization, the time-stepping method
achieves  the design-order of accuracy both for linear and nonlinear problems.

\noindent{\bf Example 3a: temporal grid refinement with pseudospectral approximation of the spatial derivative.} 
We begin with a linear advection problem $U_t  + U_x=0$
with periodic boundary conditions and 
 initial conditions $u_0(x) = 0.5 + 0.5 \sin(x)$ on the spatial domain $x \in [0,2 \pi]$.
We discretize the spatial grid with $N=41$ equidistant points and use the Fourier pseudospectral 
differentiation matrix $\mD$ \cite{HGG2007} to compute $F \approx - U_x \approx  - \mD u$. We use a Lax-Wendroff approach to 
approximate $F_t \approx U_{xx} \approx \mD^2 u$.  In this case, the solution is a sine wave, so that the 
pseudospectral method is exact. For this reason, the spatial discretization of $F$ is exact and contributes no errors, and the
second derivative $F_t$ is also exact $F_t= u_{tt} = - \mD u_t = - \mD F  = \mD^2 u = \dot{F}$. 
We use a range of time steps, $\Delta t = \lambda \Delta x$
where we pick $\lambda =0.8, 0.7, 0.6, 0.5, 0.4, 0.3, 0.2, 0.1, \mbox{and} \;  0.05$ to compute the solution to final time $T_f = 2.0$.

In Table \ref{fig:conv_spectral} we list the errors for the SSPRK3,3,  the two derivative two stage third order method in Section 3.2
with $K=\frac{1}{\sqrt{2} }$ (listed as 2s3p in the table), the unique the two derivative two stage fourth order method (2s4p),
and the the two derivative three stage fifth order method \eqref{3s5p} (3s5p). 
We observe that the design-order of each method is verified. It is interesting to note that the SSPRK3,3 method has 
larger errors than the multiderivative method 2s3p, demonstrating that the additional computation of $\dot{F}$ does improve 
the quality of the solution.

\begin{table}
\begin{center} {\small
\begin{tabular}{|r||c|c||c|c||c|c||c|c|}
\hline
 & \multicolumn{2}{c||}{ \bf{SSPRK 3,3} } &  \multicolumn{2}{c||}{ \bf{2s3p} }  & \multicolumn{2}{c||}{ \bf{2s4p} } 
 & \multicolumn{2}{c|}{ \bf{3s5p} }  \\ \hline
\bf{$\lambda$} & \bf{error} & \bf{Order}  & \bf{error} & \bf{Order} & \bf{error} & \bf{Order} & \bf{error} & \bf{Order} \\ \hline
$0.8$  & $7.99 \times 10^{-5}$   &    ---  
           & $1.86 \times 10^{-5}$   &    ---   &  
           $1.96 \times 10^{-6}$ &    ---      & 
           $6.47  \times 10^{-8}$ &    ---        \\
$0.7$  & $5.24 \times 10^{-5}$   & $3.16$
	  & $1.21 \times 10^{-5}$  & $3.17$ 
	  &$1.12 \times 10^{-6}$ & $4.17$ & 
                $3.24  \times 10^{-8}$ & $5.17$  \\
$0.6$  &  $3.27 \times 10^{-5}$   & $3.04$ &
		$7.61 \times 10^{-6}$  & $3.05$ & 
		$6.02 \times 10^{-7}$ & $4.04$ & 
                $1.49  \times 10^{-8}$ & $5.04$  \\
$0.5$  & $1.93 \times 10^{-5}$   & $2.87$ 
	  & $4.50 \times 10^{-6}$  & $2.88$ 
	  & $2.97 \times 10^{-7}$ & $3.87$  
           & $6.12  \times 10^{-9}$ & $4.87$ \\
$0.4$  & $9.70 \times 10^{-6}$   & $3.10$
	  & $2.25 \times 10^{-6}$  & $3.10$ 
	  & $1.18 \times 10^{-7}$ & $4.10$ & 
                $1.96  \times 10^{-9}$ & $5.09$  \\
$0.3$  & $4.09 \times 10^{-6}$   & $3.00$
	   & $9.50 \times 10^{-7}$  & $3.00$ 
	   &  $3.76 \times 10^{-8}$ & $3.99$ &
                 $4.66  \times 10^{-10}$ & $5.00$  \\                 
$0.2$  & $1.21 \times 10^{-6}$   & $2.99$ 
	  & $2.81 \times 10^{-7}$  & $3.00$ 
	  & $7.43 \times 10^{-9}$  & $4.01$ &
                 $6.13  \times 10^{-11}$ & $5.00$ \\
$0.1$  & $1.50 \times 10^{-7}$   & $3.01$
	  & $3.49 \times 10^{-8}$  & $3.01$ 
	  &  $4.61 \times 10^{-10}$ & $4.00$&
                 $1.90  \times 10^{-12}$ & $5.01$  \\
$0.05$ & $1.88 \times 10^{-8}$   &$2.99$ 
	   & $4.36 \times 10^{-9}$ & $3.00$ 
	   & $2.88 \times 10^{-11}$ & $3.97$ &
                  $5.97  \times 10^{-14}$ & $4.99$ \\ \hline
\end{tabular}}
\end{center}
\caption{Convergence study for Example 3a, the linear advection problem with pseudospectral differentiation of the spatial
derivatives. Here we use $N=41$ equidistant points between $(0, 2 \pi)$, and $\Delta t = \lambda \Delta x$. The solution is 
evolved forward to time $T_f=2.0$ using the explicit SSP Runge--Kutta method (SSPRK3,3),
the two-stage third order two derivative method (2s3p), the two-stage fourth order two derivative method (2s4p),
and three-stage fifth order two derivative method (3s5p).}
\label{fig:conv_spectral}
\end{table}

\noindent{\bf Example 3b: temporal grid refinement with WENO approximations of the spatial derivative.} 
Using the same problem as above we discretize the spatial grid with $N=101$ equidistant points and use 
the ninth order weighted essentially non-oscillatory method (WENO9) \cite{Balsara_Shu_2000:WENO9} 
to differentiate the spatial derivatives. 
It is interesting to note that although the PDE we solve is linear, the use of the 
nonlinear method WENO9 results in a non-linear ODE.
To evolve this ODE in time we use a range of time steps defined by  $\Delta t = \lambda \Delta x$ for $\lambda = 0.9, 0.8, 0.7, 0.6, 0.5, 0.4$
to compute the solution to final time $T_f = 2.0$
In Table \ref{fig:conv_weno9} we list the errors for the SSPRK3,3,  the two derivative two stage third order method in Section 3.2
with $K=\frac{1}{\sqrt{2} }$ (listed as 2s3p in the table), the unique the two derivative two stage fourth order method (2s4p),
and the the two derivative three stage fifth order method \eqref{3s5p} (3s5p). 
We observe that the design-order of each method is verified and that once again the SSPRK3,3 method has 
larger errors than the multiderivative method 2s3p. These results indicate that the approximation of the second derivative term
$F_t$ via a Lax-Wendroff procedure and discretization in space does not affect the observed order of the time-stepping method,
as long as the spatial errors do not dominate. 

To see what happens if the spatial discretization errors dominate over the time errors, we compare two different WENO spatial discretizations:
the fifth order  method WENO5 and the ninth order method WENO9. 
Here, we use  $N=301$ points in space, and we choose $\Delta t = \lambda \Delta x$ for
$\lambda = 0.8, 0.6, 0.4, 0.2, 0.1, 0.05$. We evolve the solution in time to $T_f =2.0$ using the 
 2s3p and SSPRK3,3 methods. The log-log graph of the  errors vs. $\Delta t$ is given in Figure \ref{fig:conv_weno5_weno9a}
We observe that the errors from both time discretizations with the WENO5 spatial discretization (dotted lines) 
have the correct orders when $\Delta t$ is 
larger and the time errors dominate, but as $\Delta t$ gets smaller the spatial errors dominate convergence is lost.
On the other hand, for the range of $\Delta t$ studied, the time-errors dominate over the spatial errors when using the
highly accurate WENO9 (solid line) and so convergence is not lost.  We note that to see this behavior, the spatial
approximation must be sufficiently accurate. In this case this is attained by using $301$ points in space
and a high order spatial discretization.
We also studied this problem with a co-refinement of the spatial and temporal grids, where we use 
$\Delta t = 0.8 \Delta x$ for $ \Delta x = \frac{2 \pi}{N-1}$ and $N=41, 81, 161, 321$.
Figure \ref{fig:conv_weno5_weno9b} shows that while for larger $\Delta t$ the errors for WENO5 are larger than for WENO9,
and this is more pronounced for the multi-derivative methods than for the explicit Runge--Kutta,
once the grid is sufficiently refined the temporal error dominates and we see the third order convergence in time.

It is interesting to note that in all these studies  the explicit SSP Runge--Kutta
method SSPRK3,3 and the two-derivative method 2s3p behave similarly, 
indicating that the  approximation of the
second derivative does not play a role in the loss of accuracy. 


\begin{figure}[t]
    \centering
    \begin{minipage}{.475\textwidth}
        \centering
        \includegraphics[width=.9\textwidth]{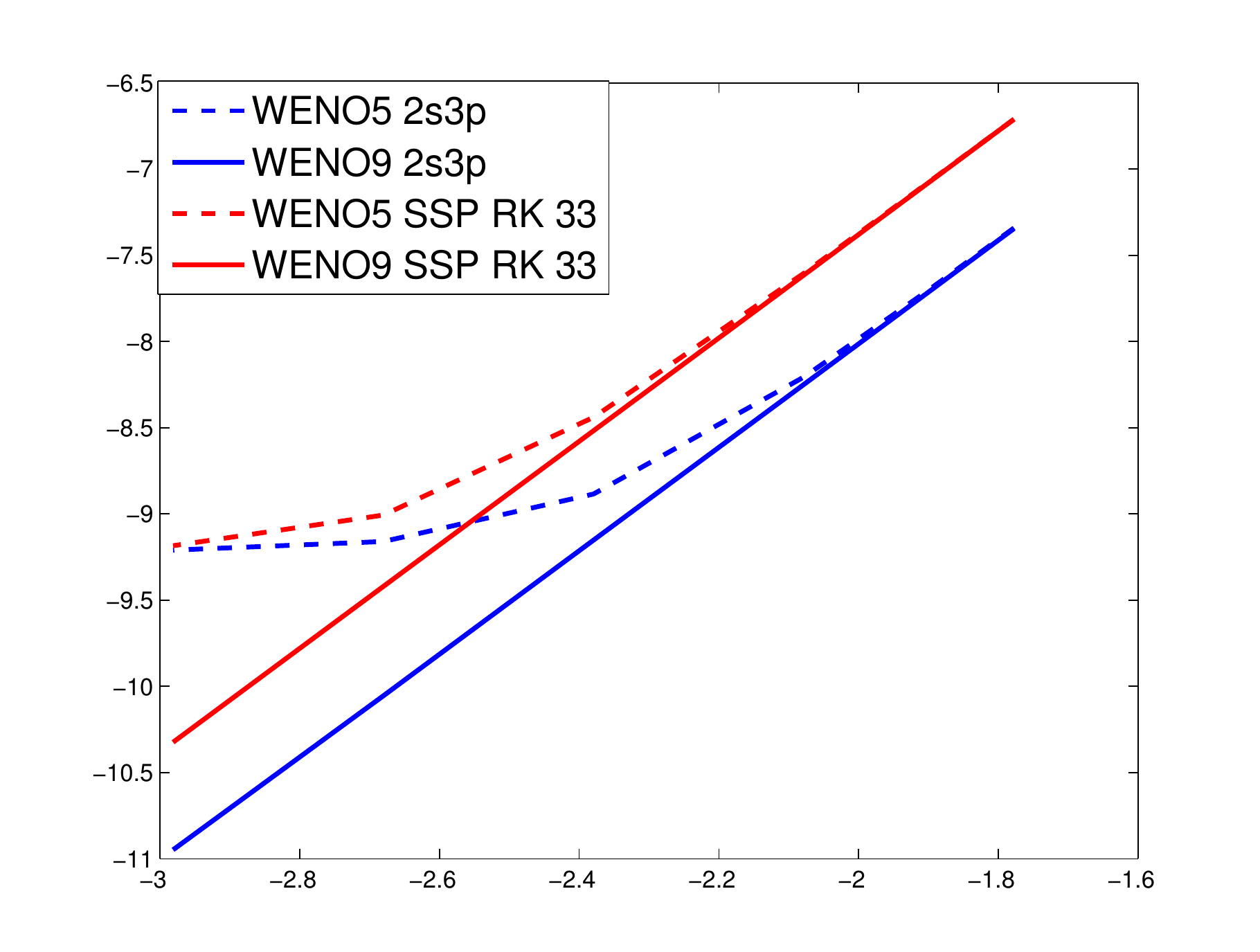}
       \caption{\small WENO5 vs WENO9 with grid refinement in time but not space Example 3b. On the x-axis are  
       $log_{10}(\Delta t)$ and on the y-axis are $ log_{10}(error)$. 
       The errors from both time discretizations with the WENO5 spatial discretization (dotted lines) 
have the correct orders when $\Delta t$ is 
larger and the time errors dominate, but as $\Delta t$ gets smaller the spatial errors dominate convergence is lost.
       }
               \label{fig:conv_weno5_weno9a}
                \end{minipage}%
    \hspace{0.2in}
    \begin{minipage}{0.475\textwidth}
        \centering
        \includegraphics[width=0.9\linewidth]{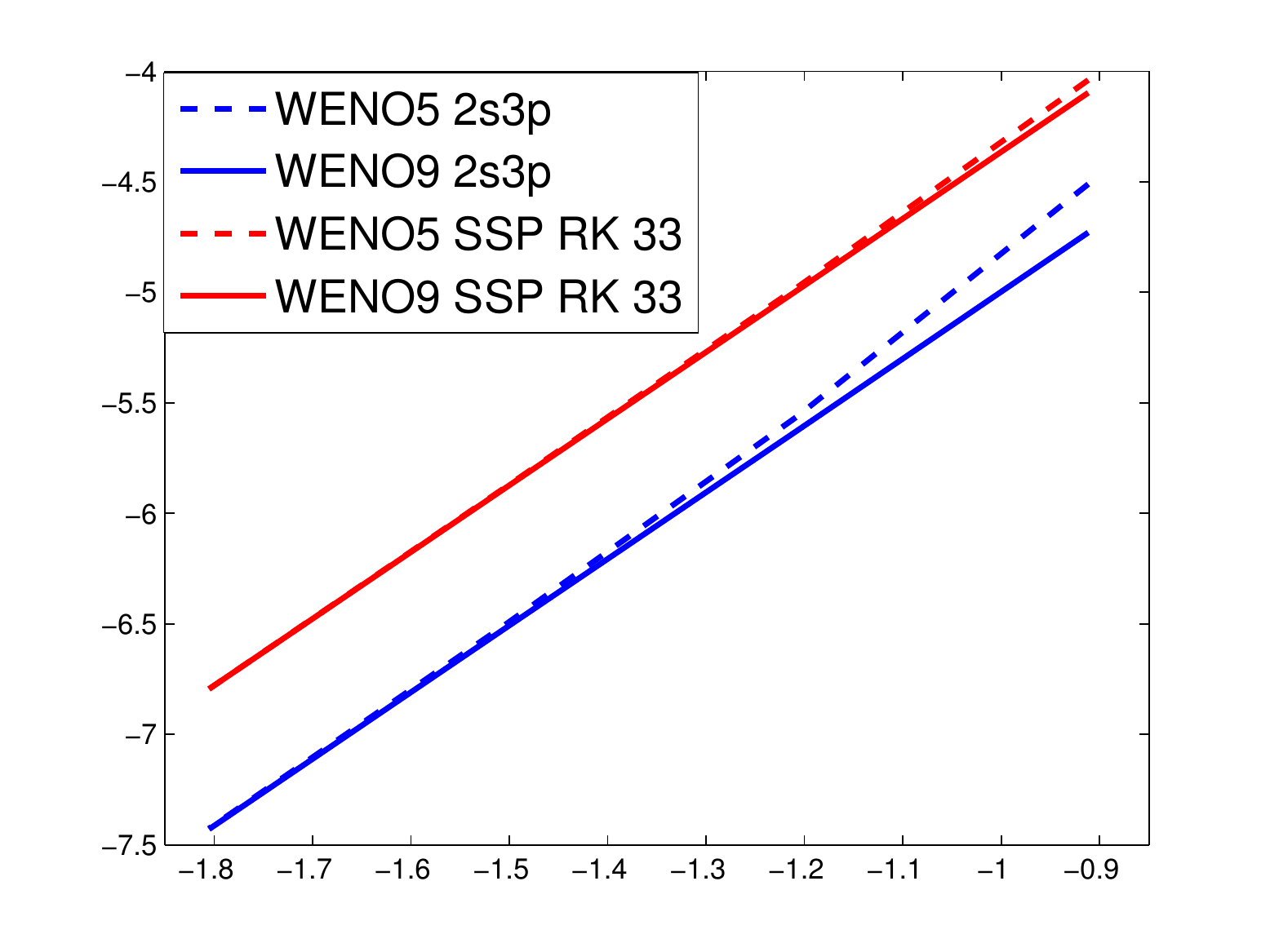}
        \caption{\small Example 3b: WENO5 vs WENO9 with refinement in both space and time, $\Delta t = 0.8 \Delta x$ . On the x-axis are  
       $log_{10}(\Delta t)$ and on the y-axis are $ log_{10}(error)$.  Here we see that 
       for larger $\Delta t$ the errors for WENO5 are larger than for WENO9 (especially for the multiderivative method)
       but once the grid is sufficiently refined
       the temporal order dominates and third order convergence in time is observed.}
               \label{fig:conv_weno5_weno9b}
                \end{minipage}%
\end{figure}

\begin{table}
\begin{center} {\small
\begin{tabular}{|r||c|c||c|c||c|c||c|c|}
\hline
 & \multicolumn{2}{c||}{ \bf{SSPRK 3,3} } &  \multicolumn{2}{c||}{ \bf{2s3p} }  & \multicolumn{2}{c||}{ \bf{2s4p} } 
 & \multicolumn{2}{c|}{ \bf{3s5p} }  \\ \hline
\bf{$\lambda$} & \bf{error} & \bf{Order}  & \bf{error} & \bf{Order} & \bf{error} & \bf{Order} & \bf{error} & \bf{Order} \\ \hline
$0.9$    &$ 7.37 \times 10^{-6} $ &    ---   &  
           $1.71  \times 10^{-6}$  &    ---   & 
         $ 8.25 \times 10^{-8}$  &    ---   &  
         $1.24 \times 10^{-9}$ &    ---      \\
$0.8$& $5.24 \times 10^{-6}$ &$2.89$&
  	$1.22 \times 10^{-6}$ &$2.89$&
	$5.21 \times 10^{-8}$ &$3.89$&
	$6.99 \times 10^{-10}$ &$4.89$\\
$0.7$ & $3.44 \times 10^{-6}$ &$3.14$ &
 	$7.99 \times 10^{-7}$ &$3.14$&
	$3.00 \times 10^{-8}$ &$4.14$&
	$3.52 \times 10^{-10}$ &$5.14$ \\
$0.6$& $2.18 \times 10^{-6}$ &$2.96$&
	$5.08 \times 10^{-7}$  &$2.96$&
	$1.63 \times 10^{-8}$  &$3.96$&
	$1.64 \times 10^{-10}$ &$4.96$\\
$0.5$& $1.27 \times 10^{-6}$ &$2.98$&
	$2.94 \times 10^{-7}$ &$2.98$&
	$7.89 \times 10^{-9}$ &$3.98$&
	$6.61 \times 10^{-11}$&$4.97$\\
$0.4$& $6.47 \times 10^{-7}$ &$3.01$&
	$1.50 \times 10^{-7}$ &$3.01$&
	$3.22 \times 10^{-9}$ &$4.01$&
	$2.18 \times 10^{-11}$&$4.97$\\ \hline
\end{tabular}}
\end{center}
\caption{Convergence study for Example 3b, the linear advection problem with WENO9 differentiation of the spatial
derivatives. Here we use $N=101$ equidistant points between $(0, 2 \pi)$, and $\Delta t = \lambda \Delta x$.}
\label{fig:conv_weno9}
\end{table}

\noindent{\bf Example 4: co-refinement of the spatial and temporal grids on linear advection with WENO7}

\noindent{\bf Example 4a:} In this example, we compare the errors and order of convergence of the multiderivative methods on 
a linear and nonlinear problem. For the linear problem, we use the linear advection problem 
\[ U_t + U_x = 0 \; \; \; \; \mbox{on $ x \in [-1,1]$} \]
with the initial condition
\begin{equation}
\label{eqn:smooth_advection_ic}
    u_0(x) = 0.5 + 0.5 \sin( \pi x ), \quad x \in [-1,1],
\end{equation}
We compute both $F$ and $\dot{F}$ using a seventh order WENO (WENO7)
\cite{Balsara_Shu_2000:WENO9}
spatial derivative. We set   $\dt = 0.8 \Delta x$ and evolve the solution to 
final time of $T_{final} = 2.0$, at which point the exact solution
is identical to the initial state. For time-stepping, we use the same three multiderivative schemes
2s3p, 2s4p, and 3s5p and the explicit SSP Runge--Kutta method SSPRK3,3. In Table  \ref{fig:advection-convergence},
where we compare the errors and orders of these three methods.
These numerical experiments show that once the mesh is sufficiently refined, the design-order of accuracy is reached.
The results verify that the proposed schemes are genuinely high-order accurate
despite the fact that the second derivative is not an exact second derivative
of the method of lines formulation.

\begin{table}
\begin{center}
{\small
\begin{tabular}{|r||c|c||c|c||c|c||c|c|} \hline
 & \multicolumn{2}{c||}{ \bf{SSPRK 3,3} } &  \multicolumn{2}{c||}{ \bf{2s3p} }  & \multicolumn{2}{c||}{ \bf{2s4p} } 
 & \multicolumn{2}{c|}{ \bf{3s5p} }  \\ \hline
\bf{$N$} & \bf{error} & \bf{order}  & \bf{error} & \bf{order} & \bf{error} & \bf{order} & \bf{error} & \bf{order} \\ \hline
$  41$ & $3.00\times 10^{-04}$ & ---  & $5.94\times 10^{-05}$ &  ---  & $7.54\times 10^{-06}$ & $---$ & $2.59\times 10^{-06}$ & ---\\
$  81$ & $3.75\times 10^{-05}$ & $3.00$ & $7.39\times 10^{-06}$ & $3.01$ & $4.71\times 10^{-07}$ & $4.00$ & $8.03\times 10^{-08}$ & $5.01$\\
$ 161$ & $4.69\times 10^{-06}$ & $3.00$ & $9.23\times 10^{-07}$ & $3.00$ & $2.94\times 10^{-08}$ & $4.00$ & $2.51\times 10^{-09}$ & $5.00$\\
$ 321$ & $5.86\times 10^{-07}$ & $3.00$ & $1.15\times 10^{-07}$ & $3.00$ & $1.84\times 10^{-09}$ & $4.00$ & $7.82\times 10^{-11}$ & $5.00$\\
$ 641$ & $7.32\times 10^{-08}$ & $3.00$ & $1.44\times 10^{-08}$ & $3.00$ & $1.15\times 10^{-10}$ & $4.00$ & $2.45\times 10^{-12}$ & $5.00$\\
$1281$ & $9.15\times 10^{-09}$ & $3.00$ & $1.80\times 10^{-09}$ & $3.00$ & $7.19\times 10^{-12}$ & $4.00$ & $7.75\times 10^{-14}$ & $4.98$\\
\hline
\end{tabular}}
\end{center}
\caption{Convergence study for Example 4a: linear advection with WENO7 spatial differentiation and
refinement of both spatial and temporal grids. Here we use $N$ points in space and 
$\Delta t = 0.8 \Delta x$. The four methods used to evolve the solution to time $T_f=2.0$ are
the explicit SSP Runge--Kutta method SSPRK3,3 and the three multiderivative methods
2s3p, 2s4p, and 3s5p.
Despite the fact that the second derivative $\dot{F}$
is not the exact second derivative $F_t$, we still observe
high-order accuracy for all methods once the grids are sufficiently refined.
\label{fig:advection-convergence}}
\end{table}

\begin{table}
\begin{center}
{\small
\begin{tabular}{|r||c|c||c|c||c|c||c|c|} \hline
 & \multicolumn{2}{c||}{ \bf{SSPRK 3,3} } &  \multicolumn{2}{c||}{ \bf{2s3p} }  & \multicolumn{2}{c||}{ \bf{2s4p} } 
 & \multicolumn{2}{c|}{ \bf{3s5p} }  \\ \hline
\bf{$N$} & \bf{error} & \bf{order}  & \bf{error} & \bf{order} & \bf{error} & \bf{order} & \bf{error} & \bf{order} \\ \hline
$ 161$ & $3.09\times 10^{-04}$ & --- & $1.91\times 10^{-04}$ & --- & $1.58\times 10^{-04}$ & --- & $1.67\times 10^{-04}$ & --- \\
$ 321$ & $5.11\times 10^{-05}$ & $2.60$ & $2.00\times 10^{-05}$ & $3.25$ & $1.33\times 10^{-05}$ & $3.57$ & $1.76\times 10^{-05}$ & $3.25$\\
$ 641$ & $6.96\times 10^{-06}$ & $2.88$ & $1.70\times 10^{-06}$ & $3.56$ & $7.12\times 10^{-07}$ & $4.22$ & $9.34\times 10^{-07}$ & $4.23$\\
$1281$ & $8.80\times 10^{-07}$ & $2.98$ & $1.81\times 10^{-07}$ & $3.23$ & $3.38\times 10^{-08}$ & $4.39$ & $2.73\times 10^{-08}$ & $5.09$\\
$2561$ & $1.10\times 10^{-07}$ & $3.00$ & $2.22\times 10^{-08}$ & $3.03$ & $1.99\times 10^{-09}$ & $4.09$ & $7.49\times 10^{-10}$ & $5.19$\\
$5121$ & $1.38\times 10^{-08}$ & $3.00$ & $2.76\times 10^{-09}$ & $3.01$ & $1.24\times 10^{-10}$ & $4.01$ & $2.21\times 10^{-11}$ & $5.08$\\
$10241$ & $1.72\times 10^{-09}$ & $3.00$ & $3.44\times 10^{-10}$ & $3.00$ & $7.75\times 10^{-12}$ & $4.00$ & $7.04\times 10^{-13}$ & $4.97$\\
\hline
\end{tabular}}
\end{center}
\caption{
Convergence study for Example 4b: Burgers' equation with WENO7 spatial differentiation and
refinement of both spatial and temporal grids. Here we use $N$ points in space and 
$\Delta t = 0.8 \Delta x$. The four methods used to evolve the solution to time $T_f=2.0$ are
the explicit SSP Runge--Kutta method SSPRK3,3 and the three multiderivative methods
2s3p, 2s4p, and 3s5p. A more refined grid is needed in this example compared to the 
linear example, but for a sufficiently refined grid 
we still observe high-order accuracy for all methods.
\label{fig:burger-convergence}}
\end{table}

\noindent{\bf Example 4b:} Next, we present results for the more difficult non-linear Burgers equation, with initial conditions prescribed by
\begin{equation}
\label{eqn:smooth_burgers_ic}
    U_0(x) = 1.0 + 0.2 \sin( \pi x ), \quad x \in [-1,1],
\end{equation}
and periodic boundary conditions.  
Once again, $F$ and $\dot{F}$ are computed via the WENO7 spatial discretization.
We use a constant time-step $\dt = 0.8 \frac{ \Delta x }{ \max_i | u_0(x_i) | }$ and 
run this problem with a final time of $T_{final} = 1.4$, before a shock forms.
Since the solution at this  point the solution  remains smooth, we can 
use the method of characteristics  to compute the  the exact solution by
\begin{align} \label{eqn:burger-implicit}
    U(t,x) = U_0\left( x - t \cdot U_0( \xi ) \right), \quad \xi = x - t \cdot U_0(\xi).
\end{align}
We solve for the implicit variable $\xi$ using Netwon iteration with a
tolerance of $10^{-14}$.
The errors and order are  presented in Table
\ref{fig:burger-convergence}.  We see that it takes an even smaller  mesh size than the linear problem 
before the errors reach the asymptotic regime,
but that once this happens we achieve the expected order.
Again, these numerical experiments further
validate the fact that if the spatial error is not allowed to dominate, the high-order design-accuracy of the time discretization
attained despite the fact that we do not directly differentiate the method of
lines formulation to define the second derivative.

\section{Conclusions}
With the increasing popularity of multi-stage multiderivative methods for use as time-stepping methods for 
hyperbolic problems, the question of their strong stability properties needs to be addressed. In this work
we presented an SSP formulation for multistage two-derivative methods. We assumed that, in addition to the 
forward Euler condition, the spatial discretization of interest satisfies a second derivative condition of the 
form \eqref{vanishing}. With these  assumptions in mind, we formulated an optimization problem which enabled us
to find optimal explicit SSP multistage two-derivative methods of up to order five, thus breaking the SSP
order barrier for explicit SSP Runge--Kutta methods. Numerical test cases verify the convergence of these methods at the 
design-order, show that sharpness of the 
SSP condition in many cases, and demonstrate the need for SSP time-stepping methods in simulations
where the spatial discretization is specially designed to satisfy certain nonlinear stability properties.
Future work will involve building SSP multiderivative methods while assuming  different base conditions (as in Remark 1)
and with higher derivatives. Additional work will involve developing new spatial discretizations suited for use with SSP multiderivative
time stepping methods. These methods will  be based on WENO or discontinuous Galerkin methods and will 
satisfy pseudo-TVD and similar properties for systems of equations.

{\bf Acknowledgements} This work was supported by: AFOSR grants  FA9550-12-1-0224, FA9550-12-1-0343, 
FA9550-12-1-0455, FA9550-15-1-0282, and FA9550-15-1-0235; NSF grant DMS-1418804;  
New Mexico Consortium grant  NMC0155-01; and  NASA  grant NMX15AP39G.

\appendix
\section{Coefficients of three stage fourth order methods}
\label{Appendix3s4pMethods}

\begin{enumerate}
\item For $K=\frac{1}{2}$ we obtain an SSP coefficient $\sspcoef=r =1.1464$. The Butcher array coefficients are given by
\[ A = \left[ \begin{array}{lll}
   0 & 0 &0\\ 0.436148675945340     & 0 &   0 \\  0.546571371212865 &  0.156647174804152  &  0   \end {array} \right], 
     \; \; \; \;   b = \left[ \begin{array}{l}
 0.528992280543542 \\   0.105732787708912 \\  0.365274931747546  \end {array} \right] \]

\[ \hat{A} = \left[ \begin{array}{lll}
   0 & 0 &  0 \\ 0.095112833764436 &   0 &  0 \\ 0.071032477596813   & 0.107904226252921  &  0  \end {array} \right] , 
       \; \; \; \; 
              \hat{b} = \left[ \begin{array}{lll}
 0.074866026156687  \\ 0.073410341982927  \\ 0.048740310097159  \end {array} \right]. \]
 
 The Shu-Osher arrays are $R \ve = (1,0,0,0)^T$,
\[ P = \left[ \begin{array}{l l l l}
  0&0&0&0  \\    0.5000& 0& 0&0 \\  0.253176729307242  & 0.179580018745470  & 0& 0\\ 
  0.181986129712082 &  0.000000001889650 &  0.418750476492704 &    
    0
\\ \end {array} \right] \]

\[  Q = \left[ \begin{array}{l l l l}
  0&0&0&0  \\    0.5& 0& 0&0 \\   0 &0.567243251947287  &0  &0\\ 0.140002406985210  & 0.003037359947341&   0.256223624973014 &0\ \end {array} \right] \]
 
\item For $K=\sqrt{\frac{1}{2}}$ we obtain an SSP coefficient $\sspcoef = r= 1.3927$
Butcher formulation
  \[ A = \left[ \begin{array}{lll}
   0 & 0 &0\\ 0.443752012194422     & 0 &   0 \\0.543193299768317 &  0.149202742858795 &  0   \end {array} \right],
     \; \; \; \;   
     b = \left[ \begin{array}{l}
 0.515040964378407 \\  0.178821699719783 \\  0.306137335901811  \end {array} \right] \]
     
 \[ \hat{A} = \left[ \begin{array}{lll}
   0 & 0 &0\\ 0.098457924163299 &   0 &  0 \\ 0.062758211639901   &0.110738910914425  &  0  \end {array} \right],
     \; \; \; \;   
 \hat{b} = \left[ \begin{array}{l}
 0.072864982225864  \\ 0.073840478463180 \\  0.061973770357455
   \end {array} \right].\]

 The Shu-Osher arrays are $R \ve = (1,0,0,0)^T$,
    \[
  P = \left[ \begin{array}{l l l l}
  0&0&0&0  \\    0.618033988749895& 0& 0&0 \\   0.362588515112176 &  0.207801573327953     & 0& 0\\  0.144580879241747&   0.110491604448675 &  0.426371652664792 & 0\\ \end {array} \right] \]
\[
  Q = \left[ \begin{array}{l l l l}
  0&0&0&0  \\    0.381966011250105& 0& 0&0 \\   0 & 0.429609911559871  &0  &0\\ 0.078129569197367   &0&  0.240426294447419 &0\ \end {array} \right] .\]

\item For $K=1$ we obtain an SSP coefficient $\sspcoef=r =1.6185$. The Butcher array coefficients are given by
  \[ A = \left[ \begin{array}{lll}
   0 & 0 &0  \\  0.452297224196082     & 0 & 0  \\ 0.528050722182308& 0.159236998008155  &  0     \end {array} \right],
     \; \; \; \;   
     b = \left[ \begin{array}{l}
  0.502519798444212\\   0.210741084344740 \\  0.286739117211047  \end {array} \right] \]
     
 \[ \hat{A} = \left[ \begin{array}{lll}
   0 & 0 &0  \\
    0.102286389507741 &   0 &  0  \\  
   0.055482128781494 & 0.108677624192402 &  0  \end {array} \right],
    \; \; \; \;   
\hat{b} = \left[ \begin{array}{lll}
 0.071256397204544 \\  0.069475972085130 \\   0.066877749079721  \end {array} \right]. \]

 The Shu-Osher arrays are $R \ve = (1,0,0,0)^T$,
    
\[ P = \left[ \begin{array}{llll}
  0&0&0&0  \\      0.732050807568877 & 0& 0&0 \\  0.457580533944888 &  0.257727809835459     & 0& 0\\  0.137886171629970  & 0.176326540063367 &  0.464092174540814 &  0\\ \end {array} \right], \]
  \[
  Q = \left[ \begin{array}{l l l l}
  0&0&0&0  \\     0.267949192431123 & 0& 0&0 \\   0 &0.284691656219654 &0  &0\\ 0.046502314960818   &0&   0.175192798805030 &0\ \end {array} \right] . \]

\end{enumerate}


\section{Two stage third order method}
This code gives the SSP coefficient and the Butcher and Shu Osher arrays for the 
     optimal explicit SSP two stage third order method given the value $K$.

\begin{verbatim}
clear all
k=sqrt(0.5); % Choose K
tab=[];
% Set up the polynomial for Q(3,1) to solve for SSP coefficient r
AA=sqrt(k^2+2) - k;
p0=2*k*(AA-2*k) + 4*k^3*AA;
p1=-p0;
p2=(1-p0)/(2*k^2);
p3= -(p0/(2*k) + k)/(6*k^3);
CC=[p3,p2,p1,p0]; % polynomial coefficients
RC=roots(CC);
r=RC(find(abs(imag(RC))<10^-15)); % SSP coefficient is the only real root.
%--------------------------------------------------------------------------
% Once we have the k and r we want we define the method following (21)
a= (k*sqrt(k^2+2)-k^2)/r;
b2 = ((k^2*(1-1/r)) + r*(.5-1/(6*a)))/(k^2+.5*r*a);
b1=1-b2;
ahat=.5*a^2;
bhat1=.5*(1-b2*a)-1/(6*a);
bhat2=1/(6*a)-.5*b2*a;
% The Butcher arrays are given by
A=zeros(2,2); Ahat=A;
A(2,1)=a;
b=[b1,b2];
Ahat(2,1)=ahat;
bhat=[bhat1,bhat2];
S=[0 0 0 ; a 0 0; b1 b2 0];
Shat=[0 0 0 ; ahat 0 0 ; bhat1 bhat2 0];
% The Shu-Osher matrices are given by
I=eye(3);e=ones(3,1);
Ri=I+r*S+(r^2/k^2)*Shat;
v=Ri\e;
P = r*(Ri\S);
Q= r^2/k^2*(Ri\Shat);
violation=min(min([v, S, Shat, P, Q])) % use this to check that all these are positive
tab=[tab;[k,r,violation]]; % build the table of values 
     \end{verbatim}

     \section{Three stage fifth order method}
     This code gives the SSP coefficient and the Butcher and Shu Osher arrays for the 
     optimal explicit SSP three stage fifth order method given the value $K$.
   \begin{verbatim}  
clear all
format long
syms r
tab=[];
k=sqrt(0.5) %The second derivative condition coefficient K
% Find the SSP coefficient C given K
a21= 240*k^6*(1 -r - r^2/(2*k^2) + r^3/(6*k^2) + r^4/(24*k^4) - r^5/(120*k^4))/r^6;
Q31=10*r^2*a21^4 - 100*k^2*a21^3 - 10*r^2*a21^3 + 130*k^2*a21^2  + 3*r^2*a21^2 - 50*k^2*a21 +6*k^2;
RC=vpasolve(simplify(r^22*Q31)==0);
rr=RC(find(abs(imag(RC))<10^-15));
C= max(rr) %The SSP coefficient
% -------------------------------------------------------------------------
% The Butcher array coefficients given K and C
a21= 240*k^6*(1 -C - C^2/(2*k^2) + C^3/(6*k^2) + C^4/(24*k^4) - C^5/(120*k^4))/C^6;
ah32= ( (3/5 -a21)^2/(a21*(1-2*a21)^3) - (3/5 -a21)/(1-2*a21)^2 )/10;
ah31= ( (3/5 -a21)^2/(1-2*a21)^2)/2  -ah32;
a31= (3/5 -a21)/(1-2*a21);
bh2=(2*a31-1)/(12*a21*(a31-a21));
bh3=(1-2*a21)/(12*a31*(a31-a21));
bh1=1/2-bh2-bh3;
ah21=(1/24 - bh3*(ah31+ah32) )/bh2;
% Build the Butcher matrices
a= C*a21+ah21*C^2/k^2;
b= C*a31+ah31*C^2/k^2;
c= ah32*C^2/k^2;
d= C+ bh1*C^2/k^2;
e= bh2*C^2/k^2;
f = bh3*C^2/k^2;
Ri=([1 0 0 0; a 1 0 0 ; b c 1 0; d e f 1]);
S=[0 0 0 0; a21 0 0 0; a31 0 0 0 ;1 0 0 0];
Shat=[0 0 0 0; ah21 0 0 0; ah31 ah32 0 0 ;bh1 bh2 bh3 0];
% The Shu Osher matrices given C
eone=ones(4,1)
v=Ri\eone;
P = C*(Ri\S);
Q= C^2/k^2*(Ri\Shat);
% Double check that there are no violations of the SSP conditions:
violation=min(min([v, S, Shat, P, Q])) % use this to check that all these are positive
tab=[tab;[k,a21,C,violation]]; % build the table of values 
          \end{verbatim}
     \bibliography{LNL}
\end{document}